\documentclass[onefignum,onetabnum]{siamart220329}
\pdfoutput=1
\usepackage{amsfonts,comment}
\usepackage{amssymb,amsmath,url,bm,amstext}
\usepackage{mathrsfs}
\usepackage{hyperref}
\usepackage{graphicx,color,float,epstopdf}
\usepackage{mathrsfs}
\usepackage[figuresright]{rotating}
\usepackage{pst-blur,pstricks-add}
\usepackage{cases}
\usepackage{algorithm}
\usepackage{algorithmicx}
\usepackage{algpseudocode}
\usepackage{multirow}
\usepackage{amsopn}
\usepackage{mathabx}

\allowdisplaybreaks

\DeclareMathOperator{\prox}{Prox}

\newtheorem{ass}[theorem]{Assumption}
\newtheorem{example}[theorem]{Example}
\newtheorem{algo}[theorem]{Algorithm}
\newcommand{\h}[1]{\mathbf{#1}}

\DeclareMathOperator{\Rbb}{\mathbb{R}}
\DeclareMathOperator{\R}{\mathbb{R}}
\DeclareMathOperator*{\proj}{Proj}
\DeclareMathOperator*{\dom}{dom}
\DeclareMathOperator*{\inte}{int}

\newcommand{\nn}{\nonumber}

\algrenewcommand\algorithmicrequire{\textbf{Input:}}

\ifpdf
  \DeclareGraphicsExtensions{.eps,.pdf,.png,.jpg}
\else
  \DeclareGraphicsExtensions{.eps}
\fi

\newsiamremark{remark}{Remark}
\newsiamremark{hypothesis}{Hypothesis}
\crefname{hypothesis}{Hypothesis}{Hypotheses}
\newsiamthm{claim}{Claim}

\headers{Full splitting algorithm  for structured fractional programs}{Radu Ioan Bo\c t, Guoyin Li, and Min Tao}
\title{A full splitting algorithm  for fractional programs with structured numerators and denominators \thanks{Submitted to the editors DATE.
\funding{The research of RIB has been partially supported by the Austrian Science Fund (FWF), project number W1260-N35. The research of MT has been partially supported by the Natural Science Foundation of China (No.
12471289),  by the China Scholarship Council (202006195015) and it has been conducted during her research visit at the University of Vienna in 2023.}}}

\author{Radu Ioan Bo\c{t} \thanks{Faculty of Mathematics, University of Vienna, A-1090 Vienna, Austria, \email{radu.bot@univie.ac.at}.}
\and
Guoyin Li \thanks{Department of Applied Mathematics, University of New South Wales, Sydney 2052, Australia, \email{g.li@unsw.edu.au}.}
\and
{Min Tao} \thanks{School of Mathematics, National Key Laboratory for Novel Software Technology, Nanjing University, Nanjing, 210093, Republic of China, \email{taom@nju.edu.cn}}.}

\begin{document}

\maketitle

% REQUIRED
\begin{abstract}
In this paper, we
consider a class of nonconvex and nonsmooth fractional programming problems, that involve the sum of a convex, possibly nonsmooth function composed with a linear operator and a differentiable, possibly nonconvex function in the numerator and a convex, possibly nonsmooth function composed with a linear operator in the denominator. These problems have applications in various fields.  We propose an adaptive full-splitting proximal subgradient algorithm  that addresses the challenge of decoupling the composition of the nonsmooth component with the linear operator  in the numerator. We specifically evaluate the nonsmooth function in the numerator using its proximal operator of its
conjugate function. Furthermore, the smooth component in the numerator is evaluated through its gradient, and the nonsmooth in the denominator is managed using its subgradient. We demonstrate subsequential convergence toward an approximate lifted stationary point and ensure global convergence under the Kurdyka-\L ojasiewicz property, all achieved  without full-row rank assumptions on the linear operators.  We provide further discussions on {\it the tightness of the convergence results of the proposed algorithm and its related variants, and the reasoning behind aiming for an approximate lifted stationary point}.  We construct a series of counter-examples to show that the proposed algorithm and its variant might diverge when seeking exact solutions. A practical version  incorporating a nonmonotone line search is also developed to enhance its performance significantly. Our theoretical findings are validated through simulations involving limited-angle CT reconstruction and the robust sharp-ratio-type minimization problem.
 \end{abstract}
\vspace{-0.1cm}

 \begin{keywords}  structured fractional programs, full splitting algorithm, convergence analysis, lifted stationary points, Kurdyka-\L ojasiewicz property, nonmonotone line search\end{keywords}
\vspace{-0.1cm}

\begin{MSCcodes}
90C26, 90C32, 49M27, 65K05
\end{MSCcodes}

\vspace{-0.2cm}
\section{Introduction}\label{sec1}
In this paper, we consider the following class of nonsmooth and  nonconvex fractional programs:
\vspace{-0.1cm}
\begin{eqnarray}\label{PForm} \min_{\h x\in{\cal S}} F({\h x}):=\frac{g(A{\h x})+h({\h x})}{f(K{\h x})}, \end{eqnarray}
where ${\cal S}$ is a nonempty convex and compact subset of ${\mathbb R}^n$,  $f:{\mathbb R}^p\rightarrow{{\overline{\mathbb R}}:={\mathbb R}\cup\{+\infty\}}$ and $g: {\mathbb R}^s\rightarrow{\overline{\mathbb R}}$ are proper, convex and lower semicontinuous functions;  $A:{\mathbb R}^{n}\rightarrow{\mathbb R}^s$ and  $K: {\mathbb R}^{n}\rightarrow {\mathbb R}^p$ are linear operators; $h: {\mathbb R}^n \rightarrow {\mathbb R}$  is a (possibly nonconvex) differentiable function over an open set containing ${\cal S}$  and its derivative $\nabla h$ is Lipschitz continuous over this open set with a  Lipschitz constant $L_{\nabla h}$.   To ensure \eqref{PForm} is well-defined,  we assume for the denominator that $K{\h x} \in {\rm dom} f$ and $f(K{\h x})>0$ for all $\h x \in {\cal S}$. For more detailed assumptions, we direct the reader to our subsequent sections.

Problem (\ref{PForm}) falls into the category of single ratio fractional programming problems. However, it showcases a more intricate structure compared to, for instance, the problems discussed in \cite{BDL} and \cite{BC17}. In the specific scenario where the linear operators $A$ and $K$ are identity mappings (represented as $I$), the model (\ref{PForm}) simplifies to the problem addressed in \cite{LSZ}.

The model (\ref{PForm}) encompasses a variety of optimization problems stemming from diverse domains. These include the limited-angle CT reconstruction problem \cite{KS01,WTNL}, robust Sharpe ratio minimization \cite{CHZ11}, the first stage of the single-period optimal portfolio selection problem involving a risk-free asset \cite{Pang80}, the recently introduced scale-invariant sparse signal reconstruction problem \cite{Tao22,LSZ,ZengYuPong20}, and their respective extensions. Subsequently, we offer a few illustrative examples to demonstrate the nature of the model problem (\ref{PForm}).

(a)  The {\it limited-angle CT reconstruction problem} aims at reconstructing the true image from limited-angle scanning measurements, with prescribed known bounds. By representing an image as an $(n \times n)$ matrix, it can be mathematically formulated as
\begin{eqnarray}\label{CTmodel}
\min_{{\h x}\in{\cal B}}\frac{\tau\|\nabla {\h x}\|_1+\frac{1}{2}\| P{\h x}-{\h f}\|^2}{\|\nabla{\h x}\|_p},
\end{eqnarray}
where $1< p  < \infty$,  $P$ is the projection operator, ${\h f}$ the given measurement data and $\tau>0$ a regularization parameter. The linear operator ${\nabla}:{\mathbb R}^{n\times n}\mapsto {\mathbb R}^{n\times n}\times {\mathbb R}^{n\times n}$ denotes the discrete gradient operator defined as ${\nabla {\h v}} =\left(\nabla_{\h x}{\h v},\nabla_{\h y}{\h v}\right),$ where ${\nabla}_{\h x},{\nabla}_{\h y}:{\mathbb R}^{n\times n}\mapsto {\mathbb R}^{n\times n}$ are the forward horizontal
and vertical difference operators, respectively. Regarding the box constraint  ${\cal B} := [{\h l}, {\h u}] \subseteq \mathbb{R}^{n \times n}$, representing the prescribed bounds for the true image to be reconstructed \cite{WTNL}, we assume that ${\cal B} \cap {\text{span}}({\bf E}) =\emptyset$, where ${\bf E}$ is the matrix with all entries equal to one.
%This condition is imposed to exclude images where each pixel has a constant level, which does not correspond to reality.
By identifying the matrix space $\mathbb{R}^{n \times n}$ as the
Euclidean space $\mathbb{R}^{n^2}$, problem \eqref{CTmodel} can be written as a special case of
 (\ref{PForm}) with
$g({\h x}) :=\tau\|{\h x}\|_1$, $f({\h x}):=\|{\h x}\|_p$, $A=K=\nabla$, $h({\h x}):=\frac{1}{2}\| P{\h x}-{\h f}\|^2$, and ${\cal S}:= {\cal B}$. Here, $\|\cdot\|_p$ denotes the usual  $\ell_p$-norm for $1 < p < \infty$, while for $p=2$ we will simply write $\|\cdot\|_2$ for $\|\cdot\|$.

(b) The {\it robust sharp-ratio-type optimization problem} under scenario data uncertainty, which arises in finance, takes the following form:
\vspace{-0.2cm} \begin{equation}\label{muleq}
\min_{{\h x}\in \Delta} \frac{\max_{1\le i\le m_1}\{r_i-{\h a}_i^\top {\h x}\}}{\max_{1\le i \le m_2}{\h x}^\top C_i{\h x}},
\end{equation}
where $\Delta =\{{\h x} \in \mathbb{R}^n \, | \, {\h e}^\top {\h x}=1,{\h x}\ge 0\}$ with ${\h e}=(1,\ldots,1) \in \mathbb{R}^n$, $({\h a}_i, r_i)\in{\mathbb R}^n\times {\mathbb R}$, $i=1,\ldots,m_1$, are such that $r_i-{\h a}_i^\top {\h x} \ge 0$ for all $\h x \in \Delta$, and $C_i$, $i=1,\ldots,m_2$, are positive definite matrices. The standard Sharpe ratio optimization problem without data uncertainty reads as (see \cite{CHZ11}) $\max_{\h x \in{ \Delta}}\frac{{\h a}^\top{\h x}-{r}}{\sqrt{{\h x}^\top C{\h x}}}.$
 Another closely related equivalent model is
%\begin{eqnarray}\label{single}
$\max_{\h x \in{ \Delta}}\frac{{\h a}^\top{\h x}-{r}}{{{\h x}^\top C{\h x}}}$,
%\end{eqnarray}
where  $C \in \R^{n \times n}$ is a symmetric positive definite matrix and $(\h a, r) \in \mathbb{R}^n \times \mathbb{R}$.
Here, without loss of generality, we assume that  ${\h a}^\top{\h x}-{r}\ge 0$ for all  $\h x \in \Delta$. Suppose that the data $(\h a,  r)$ and $C$ are subject to scenario uncertainty, that is, $(\h a, r) \in \mathcal{U}_1=\{(\h a_1, \overline{r}_1),\ldots, (\h a_{m_1}, \overline{r}_{m_1})\} \mbox{ and } C \in \mathcal{U}_2=\{C_1,\ldots,C_{m_2}\}, $
where $({\h a}_i, \overline{r}_i)\in{\mathbb R}^n\times {\mathbb R}$, $i=1,\ldots,m_1$, are such that
${\h a}_i^\top {\h x}- \overline{r}_i \ge 0$ for all ${\h x}\in \Delta$ and $C_i$, $i=1,\ldots,m_2$, are positive definite matrices.
Then, the robust counterpart of the above Sharpe ratio optimization problem is $$\displaystyle \max_{{\h x}\in \Delta} \min_{\substack{(a,r) \in \mathcal{U}_1,\\ C \in \mathcal{U}_2}} \frac{{\h a}^\top {\h x}-r}{{\h x}^\top C{\h x}}=\max_{{\h x}\in \Delta}\frac{\min_{1\le i\le m_1}\{{\h a}_i^\top {\h x}-\overline{r}_i\}}{\max_{1\le i \le m_2}{\h x}^\top C_i{\h x}},$$ which can be further equivalently rewritten as
%\begin{eqnarray} \label{mul} \max_{\h x\in\Delta}\frac{\min_{1\le i\le m_1} \{{\h a}_i^\top {\h x}-r_i\}}{\max_{1\le i\le m_2} \sqrt{{\h x}^\top C_i {\h x}}}, \end{eqnarray}
%(\ref{single}) is a special case of (\ref{mul}), we focus on the second model (\ref{mul}).
%For (\ref{single}), it corresponds to (\ref{PForm}) with $h({\h x}):=({\h x}^\top C{\h x})^{1/2}$, $g({\h x}):=0$ and ${\h r}:=(r_1,\cdots,r_{m_1})^\top$,
%$f({\h x}) :={\h x}-{\h r}$ with $K:={\h a}^\top$ and  ${\cal S}:=\{{\h x}|{\h e}^\top {\h x}=1,\;{\h x}\ge {\bf 0} \}$;
%The problem (\ref{mul}) can be equivalently rewritten as
$\displaystyle \min_{{\h x}\in \Delta} \frac{\max_{1\le i\le m_1}\{\overline{r}_i-{\h a}_i^\top {\h x}\}}{\max_{1\le i \le m_2}{\h x}^\top C_i{\h x}}$. By adding a positive constant if necessary (without affecting the solutions), we  obtain  (\ref{muleq}).

Direct verification shows that (\ref{muleq}) is a special case of our model problem (\ref{PForm}) with $f({\h x}_{\bf 1},\cdots,{\h x}_{m_2}):=\max_{1 \le i \le m_2}\|{\h x}_i\|^2_{2}$, $K: \h x \mapsto (C_1^{1/2}\h x,\cdots,C^{1/2}_{m_2} \h x)$, $g(\h x):=\|{\h r}-{\h x}\|_{\infty}$ with ${\h r}:=(r_1,\ldots,r_{m_1})$, $A: \h x \mapsto ({\h a}_1^\top \h x,{\h a}_2^\top \h x,\cdots,{\h a}_{m_1}^\top \h x)^\top$,
  $h({\h x})=0$ and ${\cal S} :=\Delta$.
%Alternatively, one can take $f({\h x}_{\bf 1},\cdots,{\h x}_{m_2}):=\max_{1 \le i \le m_2}\|{\h x}_i\|_{2}^2$, $K: \h x \mapsto (C_1^{1/2}\h %
%x,\cdots,C^{1/2}_{m_2} \h x)$, $g(\h x):=\|\h r-{\h x}\|_{\infty}$, $A: \h x \mapsto ({\h a}_1^\top \h x,{\h a}_2^\top \h x,\cdots,{\h a}_{m_1}^\top \h x)^\top$ and%
%$\h r:=(r_1,\cdots, r_{m_1})^\top$,  $h({\h x})=0$ and ${\cal S} =\Delta$. As $f$ is super-coercive, so, ${\rm dom}f^*$ is the whole space (I %remember this fact from Heinz's book, but one can double check), which means that $f^*$ is a finite-valued convex function, and so, is locally Lipschitz function, in particular, calm.%

 %In this case $f$ is strongly convex as each $C_i$ is positive definite (and so, $f^*$ is calm on every point).

%Alternatively, one can take $f({\h x}_{\bf 1},\cdots,{\h x}_{m_2}):=\max_{1 \le i \le m_2}\|{\h x}_i\|_{2}^2$, $K: \h x \mapsto (C_1^{1/2}\h %
%x,\cdots,C^{1/2}_{m_2} \h x)$, $g(\h x):=\|\h r-{\h x}\|_{\infty}$, $A: \h x \mapsto ({\h a}_1^\top \h x,{\h a}_2^\top \h x,\cdots,{\h a}_{m_1}^\top \h x)^\top$ and%
%$\h r:=(r_1,\cdots, r_{m_1})^\top$,  $h({\h x})=0$ and ${\cal S} =\Delta$. As $f$ is super-coercive, so, ${\rm dom}f^*$ is the whole space (I %remember this fact from Heinz's book, but one can double check), which means that $f^*$ is a finite-valued convex function, and so, is locally Lipschitz function, in particular, calm.%

The conventional approach to tackling single ratio fractional programming problems commonly involves utilizing Dinkelbach's method or its variants  \cite{DW67,IB83}. The recent monograph \cite{CuiPang} comprehensively explores Dinkelbach's algorithm, incorporating surrogation as a mechanism to overcome the inherent nonconvexity of the resultant subproblems. However, solving the problem (\ref{PForm})   via Dinkelbach's method (or its variants)
might prove to be expensive and challenging due to the presence of
linear compositions both in numerator and denominator.

For solving simple single ratio problems, where compositions of nonsmooth functions with linear operators do not occur, various splitting algorithms have been proposed in a series of recent works \cite{BDL,BDL23,BC17,LSZ,ZLSIAM}. These methods share the feature that at every iterative step, instead of invoking an inner loop aimed at solving the resulting Dinkelbach's scalarization of the fractional program, they execute only one iteration of a suitable splitting algorithm and update the sequence of function values.  On the other hand, direct adaptations of these techniques for solving (\ref{PForm}) often lead to double-loop algorithms again, because of the presence of
linear compositions.

Our objective is to develop a single-loop, fully splitting algorithm with global convergence guarantees under reasonable assumptions (such as the Kurdyka-\L{}ojasiewicz (KL)  property) for solving the problem (1.1) efficiently. Here, full splitting means that one relies on only the proximity operators of  $g$ or $g^*$, and $f$ or $f^*$.
To tackle this challenge, we take inspiration from recent works such as \cite{SB19,LSZ}, developing an {\it adaptive full splitting proximal subgradient algorithm with an extrapolated step} (Adaptive FSPS).   In this approach, the linear operator in the numerator is evaluated through forward assessments. Using the technique of conjugate unfolding, it is disentangled from the nonsmooth component, which is evaluated via its proximal operator. Furthermore, the smooth component in the numerator is evaluated through its gradient, the nonsmooth component in the denominator is managed using its subgradient, and the linear operator in the denominator is also assessed through forward evaluations.

The {\it adaptive strategy},  involves a {\it backtracking approach} implemented to ensure the preservation of positivity in the sequence of augmented function values. This property plays a crucial role in the convergence analysis. We incorporate an {\it extrapolated step} \cite{Luo2022,ZYZ22} to update the primal sequence, differing from the inertial steps in  \cite{BDL,BDL23}. This step fosters positivity in the sequence of augmented function values and enhances the convergence performance of the algorithm.
 We prove the whole sequential convergence to an exact lifted stationary point under the merit function with the KL property when $g$ is smooth. In the case of $g$  nonsmooth, we prove \textit{subsequential convergence} toward an approximate lifted stationary point and establish \textit{global convergence} under a different merit function with the KL property.
Additionally, we provide insight into why we turn to an approximate lifted stationary point when $g$ is nonsmooth.
 To substantiate this reasoning, we construct  counter-examples demonstrating that the proposed algorithm, when pursuing an exact stationary solution, can easily diverge if $g$ is nonsmooth, when the smoothing parameter ($\gamma_k$) converges to zero.

Therefore, our theoretical results are tight under our proposed philosophy: to solve the problem (1.1) using a single-loop, fully splitting algorithm with global convergence under the KL property. Unlike existing primal-dual splitting approaches for nonconvex problems \cite{BCN19,HongLuoRazaviyayn16,LiPong15,PTSIAM,WangYinZeng15}, the global convergence of Adaptive FSPS is established without imposing full-row rank assumptions on the involved linear operators.
Additionally, we devise a practical version  by integrating a nonmonotone line search \cite{WNF09, ZLSIAM} to enhance its performance via adopting larger step sizes.

The rest of this paper is structured as follows. Section \ref{sec2} presents the necessary preliminary notions and results, while Section \ref{sec3} introduces the stationarity notions used in this paper and explores their relationships.
In Section \ref{sec4}, we introduce Adaptive FSPS and discuss the benefits of integrating the extrapolated step and adopting an adaptive version. Section \ref{sec5} and Section \ref{sec6} are dedicated to establishing subsequential and global convergence properties of Adaptive FSPS under the KL assumption, respectively. In Section \ref{sect7} we discuss several relevant aspects related to the conceptual algorithm FSPS.
Section \ref{sec8} introduces a practical adaptation of Adaptive FSPS by integrating nonmonotone line search, aiming to reduce dependence on unknown parameters, and conduct numerical experiments demonstrating the promising performance. Lastly, Section \ref{sec9} contains our conclusions.

\section{Preliminaries and calculus rules}\label{sec2}
 Finite-dimensional spaces within the paper will be equipped with the Euclidean norm, denoted by $\|\cdot\|$, while $\langle \cdot, \cdot \rangle$ will represent the Euclidean scalar product. Given a set ${\cal C} \subseteq {\mathbb R}^n$,  ${\text{ri}}({\cal C})$ and ${\text{int}}({\cal C})$ denote its {\it relative interior} and its {\it interior}, respectively. The function $\iota_{\cal C} : {\mathbb R}^n \rightarrow {\overline{\mathbb R}}:={\mathbb R}\cup\{+\infty\}$, defined by $\iota_{\cal C}({\h x}) = 0$, for ${\h x} \in {\cal C}$, and $\iota_{\cal C}({\h x}) = +\infty$, otherwise, denotes the {\it indicator function} of the set ${\cal C}$.

For a function $f:{\mathbb R}^n\rightarrow{\overline{\mathbb R}}$, we denote by $\dom f:=\{{\h x} \in {\mathbb R}^n: f({\h x}) < +\infty \}$ its {\it effective domain} and say that it is {\it proper} if $\dom f \neq \emptyset$. For $\overline{\h x} \in \dom f$, the set
\begin{equation*}
		{\hat\partial} f({\overline{\h x}}):=\left\{{\h v} \in {\mathbb R}^n :{\liminf \limits_{\h x\to{\overline{\h x}}\;{\h x}\neq {\overline{\h x}}}} \frac{f(\h x)-f(\overline{\h x})-\langle {\h v},{\h x}-{\overline{\h x}}\rangle}{\|{\h x}-\overline{\h x}\|}\ge 0\right\}
\end{equation*}
is the so-called {\it Fr\'{e}chet subdifferential} of $f$ at $\overline{\h x}$. The {\it limiting subdifferential} of $f$ at $\overline{\h x}$ is defined as
\begin{equation*}
		\partial f({\overline{\h x}}):=\left\{{\h v} \in \Rbb^n \! : \! \exists\; \{{\h x}^k\} \rightarrow {\overline{\h x}}, \ f({\h x}^k)\rightarrow f(\overline{\h x}), \{{\h v}^k \} \rightarrow{\h v} \ \text{as} \ k \rightarrow +\infty, \
 {\h v}^k\in{\hat\partial} f({\h x}^k) \right\}.
		\end{equation*}
If $f$ is  proper, convex and lower semicontinuous function and $\varepsilon \geq 0$, we denote by
\begin{equation}
\partial_{\varepsilon} f (\overline{\h x}) := \{{\h v} \in \Rbb^n : f({\h x}) \geq f(\overline{\h x}) + {\h v}^\top ({\h x} - \overline{\h x}) - \varepsilon \ \forall {\h x} \in \Rbb^n \}
\end{equation}
the {\it $\varepsilon$-subdifferential} of $f$ at $\overline{\h x}$. It holds
${\h v} \in \partial_{\varepsilon} f(\overline{\h x})$ if and only if $ f^*({\h v}) + f(\overline{{\h x}}) - \langle {\h v}, \overline{\h x} \rangle \leq \varepsilon$, where
$f^* : \Rbb^n \rightarrow {\overline{\mathbb R}}$, $f^*({\h v}) = \sup_{{\h x} \in \Rbb^n} \{\langle{\h v},{\h x} \rangle -f({\h x}) \}$, denotes the {\it (Fenchel) conjugate function} of $f$. The {\it convex subdifferential} of $f$ at $\overline{\h x}$ is defined by $\partial f (\overline{\h x}):=\partial_{0} f (\overline{\h x})$. The {\it domain} of the convex subdifferential is defined as $\dom(\partial f):=\{{\h x} \in \R^n : \partial f({\h x}) \neq \emptyset\}$.
For a proper, convex lower semicontinuous function  $f:{\mathbb R}^n\rightarrow{\overline{\mathbb R}}$, its {\it proximal operator of modulus $\gamma >0$}
is defined as \cite{RockWets}
%\vspace{-0.2cm}
$$\prox_{\gamma f} : \Rbb^n \rightarrow \Rbb^n, \ \prox_{\gamma f}({\h x}) = \arg\min_{\h y} \Big\{ f({\h y}) + \frac{1}{2\gamma}\|{\h y}-{\h x}\|^2\Big\}.$$
%\vspace{-0.2cm}
We further define
$ f_{\gamma}(\h x):=f \Box\left( \frac{1}{2\gamma} \|\cdot\|^2\right)(\h x)=\inf_{\h y}\left\{f({\h y})+\frac{1}{2\gamma}\|{\h x}-{\h y}\|^2\right\}.$
%On the other hand,
%\begin{eqnarray*}  f_{\gamma}(\h x):=\left( f^{**} \Box \frac{1}{2\gamma}\|\cdot\|^2\right)(\h x)=\left(f^*+\frac{\gamma}{2}\|\cdot\|^2\right)^*(\h x)=\sup_{\h z}\left\{ \langle {\h x},{\h z}\rangle-f^*({\h z})-\frac{\gamma}{2}\|\h z\|^2\right\}. \end{eqnarray*}
%Additionally, the Moreau envelope of  $f$
%is Fr\'{e}chet differentiable on ${\mathbb R}^n$, and its gradient given by
%$\nabla (f_{\gamma})(\h x)=\frac{1}{\gamma}\left({\h x}-\prox_{\gamma f}(\h x)\right) =\prox_{\frac{1}{\gamma} f^*}\left(\frac{\h x}{\gamma}\right),$
%where the last is due to $\prox_{\gamma f}(\h x)+\gamma\prox_{\frac{1}{\gamma} f^*}(\frac{\h x}{\gamma})={\h x}.$}
A proper function $f:{\mathbb R}^n\rightarrow{\overline{\mathbb R}}$ is called {\it essentially strictly convex} if it is strictly convex on every convex subset of
$\dom(\partial f)$. %A proper and convex function $f:{\mathbb R}^n\rightarrow{\overline{\mathbb R}}$ is called {\it essentially smooth} if it satisfies the following three conditions \cite{Rock70}:
%(a) $\inte(\dom f)$ is not empty; (b) $f$ is differentiable on $\inte(\dom f)$;
%(c) for every sequence $\{{\h x}^k\} \subseteq \inte(\dom f)$ which converges to a boundary point of $\dom f$, it holds $\mathop{\lim}\limits_{k \to +\infty}\|\nabla f({\h x}_k)\|=+\infty$.
For a proper, convex and lower semicontinuous function $f$, $f$ is essentially strictly convex  if and only if its conjugate $f^*$ is essentially smooth \cite{Rock70}.
 Given a linear operator $A: {\mathbb R}^n \rightarrow {\mathbb R}^m$, we denote by $A^*: {\mathbb R}^m \rightarrow {\mathbb R}^n$ its {\it adjoint operator}. %The norm of the linear operator $A$ is given by $\|A\|=\sup\{ \|A \h x\|: \|\h x\|=1\}$.
Next, we review the Kurdyka-\L ojasiewicz (KL) property  \cite{ABRS,ABS,BDL07} and the concept of calmness \cite{RockWets}.

\begin{definition} \label{def:KL}
A proper and lower semicontinuous function
$f: \mathbb R^n\rightarrow \overline \R$ is said to satisfy the Kurdyka-\L ojasiewicz (KL) property at a point ${\hat {\h x}}\in {\text{\rm dom}}(\partial f)$ if there exist a constant $\mu \in(0,+\infty]$, an open neighborhood $U$ of ${\hat {\h x}},$ and a desingularization function $\phi:\;[0,\mu)\rightarrow[0,+\infty)$, which is continuous and concave, and continuously differentiable on $(0,\mu)$ with $\phi(0)=0$ and $\phi'>0$ on $(0,\mu)$, such that for every $\h x\in U$ with $f({\hat {\h x}})< f(\h x)< f(\hat {\h x}) +\mu$ it holds
		$\phi'(f(\h x)-f(\hat {\h x})){\text{\rm{dist}}}(\h 0,\partial f({\h x}))\ge 1.$
\end{definition}

%The  will play an important role throughout this paper.

\begin{definition}\label{calm}
A proper function $f: {\mathbb R}^n \rightarrow {\overline{\mathbb R}}$ is said to be {\it calm} at ${\h x}\in {\text{\rm dom}}f$
 if there exist $\varepsilon >0$ and $\kappa>0$ such that
$ | f({\h y})-f({\h x})|\le \kappa \|{\h y}-{\h x}\|$
for all ${\h y}\in B({\h x},\varepsilon):=\{{\h z} \in \R^n : \|{\h z} - {\h x} \| < \varepsilon \}$.
\end{definition}

The following provides the Fr\'{e}chet subdifferential formula of the ratio functions.
\vspace{-0.4cm}
\begin{lemma}\label{ratioC2}
Let $O \subseteq \R^n$ be an open set, and $f_1 : O \rightarrow \overline{\mathbb \R}$ and $f_2 : O \rightarrow \mathbb \R$ be two functions which are finite at ${\h x} \in O$ with $f_2({\h x})>0$. Suppose that $f_1$ is continuous at ${\h x}$ relative to
${\text{\rm dom}}f_1$, that $f_2$  is calm at ${\h x}$, and denote $\alpha_i:=f_i({\h x})$, $i=1,2$.

\noindent (i) Then
\begin{eqnarray}\label{Lem3:eq1} {\hat\partial} \left(\frac{f_1}{f_2}\right)({{\h x}}) =\frac{{\hat\partial} (\alpha_2 f_1-\alpha_1 f_2)({{\h x}})}{f_2({{\h x}})^2}.\end{eqnarray}
(ii) If, in addition, $f_2$ is convex  and $\alpha_1\ge0$, then
\begin{eqnarray}\label{Lem3:eq2}{\hat\partial} \left(\frac{f_1}{f_2}\right)({{\h x}}) \subseteq \frac{{\hat\partial} (\alpha_2 f_1)({{\h x}})-\alpha_1{\hat\partial} f_2({{\h x}})}{f_2({{\h x}})^2}.\end{eqnarray}
\end{lemma}
\begin{proof}
(i) The proof is similar to \cite[Proposition 2.2]{ZLSIAM}.
(ii) If $f_2$ is convex and $\alpha_1\ge0$,  then ${\hat{\partial}}(\alpha_1 f_2)({\h x})\neq {\emptyset}$ thanks to ${\h x}\in {\text{\rm int}}({\text{\rm dom}}f_2)$. According to \cite[eq. (1.6)]{MNY06}, this further leads to
${\hat\partial} (\alpha_2 f_1-\alpha_1 f_2)({{\h x}})\subseteq {\hat\partial} (\alpha_2 f_1)({\h x})-{\hat\partial}(\alpha_1 f_2)({{\h x}})={\hat\partial} (\alpha_2 f_1)({\h x})-\alpha_1{\hat{\partial}} f_2({{\h x}}). $
\end{proof}

\noindent Next, we present a lemma that will be useful in establishing approximate stationarity for the accumulation points of the sequence generated by the proposed algorithm.

  \begin{lemma}\label{epsappsol}
 Let $g:{{\mathbb R}^n}\rightarrow{\overline{\mathbb R}}$ be a proper, convex and lower semicontinuous  function, and ${\h w}\in{\text{\rm int}}({\text{\rm dom}}g)$.
Let $\varepsilon>0$ and ${\cal K}$ be a compact set such that
 $B({\h w},\varepsilon)\subseteq {\cal K} \subseteq {\text{\rm int}}({\text{\rm dom}}g)$ and $g$ is  Lipschitz continuous on  $\cal K$ with constant $\kappa >0$. Further, let ${{\h z}} \in \R^n $ be such that
 ${\text{\rm dist}}({\h w}, \partial g^*({{\h z}})):=\inf \{\|{\h w} - {\bm \eta} \| : {\bm \eta} \in \partial g^*({{\h z}})\}\le \varepsilon$.
 Then, one has
 \begin{eqnarray}\label{Lem5:con} {{\h z}}\in \partial_{\hat{\varepsilon}} g({\h w}),\end{eqnarray}
 where ${\hat\varepsilon}:=2\kappa\varepsilon$.
 Furthermore, if $g$ has full domain and $g$ is differentiable with $L_{\nabla g}$-Lipschitz continuous gradient ($L_{\nabla g} >0$), then \eqref{Lem5:con} holds for
 ${\hat\varepsilon}:=\frac{L_{\nabla g} }{2}\varepsilon^2$.
 \end{lemma}

 \begin{proof} First,  as $B({\h w},\varepsilon)\subseteq {\cal K}$ and $g$ is  Lipschitz continuous on  $\cal K$ with constant $\kappa >0$, we observe that $\sup\{\|{\bm \xi}\| : {\bm \xi} \in \partial g({\h w}+{\h u}), \|{\h u}\|\le \varepsilon\}\le \kappa.$ %Indeed, this follows since for all $\h u \in \R^n$ with $\|\h u\| \leq \varepsilon$ and all ${\bm \xi} \in \partial g({\h w}+{\h u})$ it holds
%$-{\bm \xi}^\top {\h u} \leq g({\h w})-g({\h w}+{\h u}) \leq \kappa \|\h u\|$.
For $\overline{\bm \eta}:= {\text{Proj}}_{\partial g^*({{\h z}})}({\h w})$, the projection of $\h w$ on $\partial g^*({{\h z}})$ (which exists and is unique), it holds $\|\overline {\bm \eta} -{\h w}\|\le \varepsilon$. Therefore, for $\overline{\h u}:=\overline {\bm \eta}-{\h w}$, we have $\|\overline {\h u}\|\le \varepsilon$, ${\h w}+ \overline {\h u}\in \partial g^*({\h z})$ or, equivalently,
 ${{\h z}}\in \partial g({\h w}+ \overline{\h u}).$
Next, we claim that
\begin{eqnarray}\label{lem4:eq2} g({\h w})-g({\h w}+\overline{\h u}) + \langle \overline{\h u}, {{\h z}} \rangle\le {\hat\varepsilon},\end{eqnarray}
where ${\hat\varepsilon}$ is defined in the statement of the lemma. Since
$\langle {\h w}+ \overline{\h u},{{\h z}} \rangle-g^*({\h z})=g({\h w}+\overline{\h u})$, thanks to the fact ${{\h z}}\in \partial g({\h w}+\overline{\h u})$, this is equivalent to $ g({\h w})+g^*({{\h z}})\le {\hat\varepsilon}+ \langle {\h w},{{\h z}}\rangle,$ which is nothing else than (\ref{Lem5:con}).
Now we will prove that \eqref{lem4:eq2} is true. By direct calculations, we have
$g({\h w})-g({\h w}+\overline{\h u})+\langle \overline{\h u}, {{\h z}} \rangle
\le \kappa \|\overline {\h u}\|+\kappa \|\overline{\h u}\| \leq 2\kappa\varepsilon={\hat \varepsilon}$.
If $\dom g = \R^n$, $g$ is differentiable and $\nabla g$ is Lipschitz continuous with Lipschitz constant $L_{\nabla g}  >0$, then ${{\h z}}=\nabla g({\h w}+\overline{\h u})$, and so, by the  usual Descent Lemma \cite[Lemma 1]{BST14},
$g({\h w})-g({\h w}+\overline {\h u}) + \langle \overline{\h u}, {{\h z}} \rangle \leq \frac{L_{\nabla g} }{2}\|{\overline{\h u}}\|^2 \leq \frac{L_{\nabla g} }{2}\varepsilon^2.$
\end{proof}

\section{Basic assumptions and stationary points of fractional programs}\label{sec3}

We introduce the basic assumptions regarding (\ref{PForm})  and present two versions of stationary points for fractional programs, examining their relationships.

\begin{ass}\label{ass1} Throughout this paper, we assume that\\
(a) ${\cal S} \subseteq {\mathbb R}^n$ is a nonempty convex and compact set;\\
(b) $g$ is a  proper, convex and lower semicontinuous function;\\
(c) $h$ is  differentiable with Lipschitz continuous gradient over an open set containing the compact set ${\cal S}$ with a  Lipschitz constant $L_{\nabla h}$;\\
(d) $f$ is  a proper, convex and lower semicontinuous function with   $K(\mathcal{S}) \subseteq {\rm int} ({\rm dom} f)$ and
 $f(K {\h x})>0$ for all ${\h x}\in{\cal S}$;\\
(e)  ${\cal S}\cap A^{-1} ({\text{\rm dom}}g)\neq \emptyset$  and $\alpha:=\inf_{\h x \in {\cal S}}\{g(A{\h x})+h({\h x})\}> 0.$
\end{ass}

\begin{remark}\label{trick}
The assumption ${\cal S}\cap A^{-1} ({\text{\rm dom}}g)\neq \emptyset$  is to guarantee
that the objective function $F$ is not identically $+\infty$. The second condition in Assumption \ref{ass1}(e) might seem restrictive than it is at first glance. Indeed, it can be enforced without loss of generality. To see this, note from Assumption \ref{ass1}(a)-(d) that  $\inf_{{\h x}\in{\cal S}} F({\h x}) > -\infty$. Then, we can equivalently reformulate the optimization problem by augmenting the objective with a suitable positive constant. This reformulation allows us to redefine the function $g$ and the linear operator $A$ to satisfy the requirement $\alpha > 0$. \footnote{Different choices of the constant $\alpha$ may impact the numerical performance in general.}
\end{remark}

\begin{definition} \label{statFL}  For the optimization problem (\ref{PForm}), we say that ${\overline{\h x}}\in {\mathbb R}^n$ is
\begin{itemize}
\item[(i)] a Fr\'{e}chet  stationary point if $0\in{\hat\partial}\left (\frac{g\circ A+h+\iota_{\cal S}}{f\circ K}\right)({\overline{\h x}})$;
\item[(ii)] a limiting  lifted stationary point  if
\begin{eqnarray*}
0\in \left(A^* \partial g(A{\overline{\h x}})+\nabla h(\overline{\h x})+\partial \iota_{\cal S}({\overline{\h x}})\right) f(K\overline{\h x})-(g(A\overline{\h x})+h(\overline{\h x}))K^* \partial f(K{\overline{\h x}}).\end{eqnarray*}
\end{itemize}
\end{definition}

From the definition, it is clear that any local minimizer ${\overline{\h x}}\in {\mathbb R}^n$  of (\ref{PForm})
 is a  Fr\'{e}chet  stationary point. On the other hand, if ${\overline{\h x}}\in {\mathbb R}^n$ is a Fr\'{e}chet  stationary point of (\ref{PForm}) such that $K{\overline{\h x}} \in {\text{\rm int}}({\text{\rm dom}}f)$, and either $\overline{\h x} \in {\rm ri} (\mathcal{S}) \cap A^{-1}{\rm ri}({\rm dom} g)$  or $\mathcal{S}$ is polyhedral and $\overline{\h x} \in  \mathcal{S} \cap A^{-1}{\rm ri}({\rm dom} g)$, then, according to  Lemma \ref{ratioC2}, ${\overline{\h x}}\in {\mathbb R}^n$ is also a limiting lifted stationary point of (\ref{PForm}).
 The following examples (see \cite[Example 3.1]{BDL}) illustrates that a limiting lifted stationary point may not  be a Fr\'{e}chet stationary point.

\begin{example}\label{ex3.6} Consider the  problem
$\min_{x\in[-1,1]}\frac{x^2+1}{|x|+1}.$
Then ${\cal S}=[-1,1]$, $g(x) =0$, $h(x) =x^2+1$ and $f({x})= |x|+1$, and $A=K=I$. For $\overline{ x} =0$, we have
$ \left(A^* \partial g(A{\overline{ x}})+\nabla h(\overline{ x})+\partial \iota_{\cal S}({\overline{ x}}) \right) f(K\overline{ x})-(g(A\overline{ x})+h(\overline{ x}))K^* \partial f(K{\overline{x}})=[-1,1]$.
Thus, $\overline{x}$ is a  limiting lifted stationary point.
On the other hand, using assertion (i) of Lemma \ref{ratioC2} for $\alpha_1 =\alpha_2=1$, we have
${\hat{\partial}}\left(\frac{(({\cdot})^2+1)+\iota_{[-1,1]}({\cdot})}{|\cdot|+1} \right)(\overline{ x})= \emptyset,$
therefore, $\overline{ x}$ is  not a Fr\'{e}chet  stationary point.
\end{example}

%The following examples (see \cite[Example 3.1]{BDL}) illustrates that a limiting lifted stationary point may not necessarily be a Fr\'{e}chet stationary point.

%\begin{example}\label{ex3.6} Consider the  one-dimensional fractional optimization  problem
%$$\min_{x\in[-1,1]}\frac{x^2+1}{|x|+1}.$$
%Then ${\cal S}=[-1,1]$, $g(x) =0$, $h(x) =x^2+1$ and $f({x})= |x|+1$, and $A=K=I$. For $\overline{ x} =0$, we have
%\begin{align*}
%& \left(A^* \partial g(A{\overline{ x}})+\nabla h(\overline{ x})+\partial \iota_{\cal S}({\overline{ x}}) \right) f(K\overline{ x})-(g(A\overline{ x})+h(\overline{ x}))K^* \partial f(K{\overline{x}})\nn\\
%= \ & (2 {\overline{ x}}+\partial \iota_{[-1,1]}(\overline{ x}))\times(|\overline{x}|+1)-({\overline{ x}}^2+1)\times\partial(|\cdot|+1)(\overline{ x})=[-1,1].
%\end{align*}
%Thus, $\overline{x}$ is a  limiting lifted stationary point.

%On the other hand, using assertion (i) of Lemma \ref{ratioC2} for $\alpha_1 =\alpha_2=1$, we have
%\begin{eqnarray*}{\hat{\partial}}\left(\frac{(({\cdot})^2+1)+\iota_{[-1,1]}({\cdot})}{|\cdot|+1} \right)(\overline{ x})= {\hat{\partial}}(({\cdot})^2+\iota_{[-1,1]}(\cdot)-|\cdot|) (\overline{x}) = \hat{\partial} (- |\cdot|)(\overline{ x}) = \emptyset,\end{eqnarray*}
%therefore, $\overline{ x}$ is  not a Fr\'{e}chet  stationary point.
%\end{example}

Next we introduce the notion of an approximate lifted stationary point for the problem (\ref{PForm}).

\begin{definition}\label{appsol} Given $\epsilon_1,\epsilon_2\ge 0$, we say that ${\overline{\h x}}\in \R^n$ is
a limiting $({\epsilon_1,\epsilon_2})$-lifted stationary point of the problem (\ref{PForm})  if there exists ${{\overline{\Psi}}} \in \mathbb{R}$ with $|{{\overline{\Psi}}}-(g(A\overline{\h x})+h(\overline{\h x}))|\le \epsilon_2$ such that
\begin{eqnarray*}
0\in \left (A^* \partial_{\epsilon_1} g(A{\overline{\h x}})+\nabla h(\overline{\h x})+\partial \iota_{\cal S}({\overline{\h x}}) \right)f(K\overline{\h x})-{\overline{\Psi}} \, K^*\partial f(K{\overline{\h x})}.\end{eqnarray*}
\end{definition}
If $\epsilon_1=\epsilon_2=0$, then this notion reduces to the one of a limiting lifted stationary point.

\section{Full splitting adaptive algorithm}\label{sec4}
\subsection{Conceptual algorithmic framework}\label{subsec41}

We first propose a {\it conceptual} algorithmic framework for solving (\ref{PForm}) which we call {\it full splitting proximal subgradient algorithm with  an extrapolated step} (FSPS).

Let $0<\beta<2$, the sequences of positive numbers $\{\gamma_k\}$ and $\{\delta_k\}$, $\theta_0>0$ and a given starting point $({\h x}^0, {\h z}^0, {\h u}^0)$.  For all $k \geq 0$, we consider the following update rule:
\begin{equation}\label{FSPSO}
\left \{
\begin{array}{rcl}
{\h y}^{k+1} & \in & \partial f(K{\h x}^k)\\[0.1cm]
 \h x^{k+1} & = & {\text{Proj}}_{\cal S}\left({\h u}^k+\frac{\theta_k}{\delta_k}K^* {\h y}^{k+1}-\frac{1}{\delta_k}\nabla h({\h x}^k)-\frac{1}{\delta_k}A^* {\h z}^k\right),\\[0.1cm]
{\h u}^{k+1} & := & (1-\beta) {\h u}^k+\beta {\h x}^{k+1},\\[0.1cm]
{\h{z}}^{k+1} & := &\arg\min_{\h z} \left[ g^*({\h z})- \langle A{\h x}^{k+1}, {\h z} \rangle+\frac{\gamma_k}{2}\|\h z\|^2\right],\\[0.1cm]
\theta_{k+1} & := & \displaystyle{\frac{{\rm\Psi}({\h x}^{k+1},{\h z}^{k+1},{\h u}^{k+1};\delta_k,{\gamma_k})}{f(K{\h x}^{k+1})}},
\end{array}
\right.
\end{equation}
\vspace{-0.2cm}
where
%\begin{eqnarray*}
%&&
${\Psi}({\h x},{\h z},{\h u}, \delta,\gamma)\!:=\!\langle{\h z},A{\h x} \rangle-g^*({\h z})+h({\h x})+\iota_{\cal S}({\h x}) + \frac{\delta}{2}\|{\h x}-{\h u}\|^2 \!-\! \frac{\gamma}{2}\|{\h z}\|^2$.
%\end{eqnarray*}
%\vspace{0.2cm}
\begin{remark}\label{remarkfsps}
Assumption \ref{ass1}(d) ensures that
$\partial f(K{\h x})\neq \emptyset$ when ${\h x}\in{\cal S}$. As we will see later in the adaptive version, the positive sequences $\{\gamma_k\}$ and $\{\delta_k\}$ will be chosen such that $0 < \gamma_{k+1} \leq \gamma_k$ and, respectively,
$ \delta_k:=2\nu+L_{\nabla h}+\frac{2\|A\|^2}{\gamma_{k}}$ for all $k \geq 0$, where $\nu>0.$
\end{remark}
\vspace{-0.2cm}

\begin{remark}\label{remark:4.2}
In the FSPS numerical scheme, updating ${\h z}^{k+1}$ amounts to computing the proximal operator of $\frac{1}{\gamma_k} g^*$, which can be done efficiently in many situations.
%In particular, for the two motivating examples in the introduction, closed-form solutions for the proximal operator of $\frac{1}{\gamma_k} g^*$ are indeed available.
We refer to $\{\theta_{k}\}$ as the augmented function value sequence. In fact, the update of ${\h z}^{k+1}$ can be written as ${\h z}^{k+1}={\rm prox}_{g^*/\gamma_k}(A \h x^{k+1}/\gamma_k)= \nabla g_{\gamma_k}(A \h x^{k+1})$. So,
$$\theta_{k+1}=\frac{g_{\gamma_k}(A\h x^{k+1})+h(\h x^{k+1})+\frac{\delta_k}{2}\|\h x^{k+1}-\h u^{k+1}\|^2}{f(K \h x^{k+1})} \le F(\h x^{k+1})+ \frac{\frac{\delta_k}{2}\|\h x^{k+1}-\h u^{k+1}\|^2}{f(K \h x^{k+1})}.$$
% which gives an  approximation to value of the objective function $F(\h x^{k+1}).$
 In particular, if $\beta=1$,  $\{\theta_k\}$ serves as a surrogate from below for the objective function sequence.
To establish the convergence of the algorithm, it is crucial to ensure the positivity of $\theta_k$. This is a non-trivial task and thus motivates us to develop an adaptive scheme to ensure this property.
\end{remark}
The following toy example illustrates the potential accelerating of using a large   $\beta$.
\begin{example}\label{ex2}  Consider the problem (\ref{PForm}) for  ${\cal S}={[0,1]^2}$, $A=K=I$, and $g,h,f:\mathbb{R}^2 \rightarrow \mathbb{R}$ given by
$g({\h x}) = \tau \|{\h x}\|_{1}$, $h({\h x}) = \frac{1}{2}\|B{\h x}-{\h b}\|^2$ and  $f({\h x}) = \|{\h x}\|$, respectively, where $B$ is a discrete cosine transform matrix, ${\h x}^{*} := (1,0)^{\top}$ is the true signal, ${\h b}:= B{\h x}^* =(
		0.7071,
		0.7071)^\top,$ and $\tau = 10^{-3}$. We tested FSPS for $\gamma_k := q^k$, with $q:=0.9999$, $\delta_k:=5+L_{\nabla h}+\frac{2}{\gamma_{k}}$ for $k \geq 0$, $\theta_0:=0.8053$, starting point given by ${\h u}^0 = {\h x}^0:=(0.2,1)^{\top}$ and ${\h z}^0:=(10^{-3},10^{-3})^{\top}$ and various choices for the extrapolation parameter $\beta \in \{0.2,0.6,1.0,1.4,1.8\}$. As a stopping criterion, we considered
${\rm RErr}(k):=\|{\h x}^{k} - {\h x}^{*}\|/\|{\h x}^{*}\| < 10^{-6}$.
Fig. \ref{RelMSE}  illustrates  the trajectories of the relative error ${\rm RErr}(k)$  and the objective function value generated by FSPS for different choices of $\beta$. It indicats
 that a larger value of $\beta$ might lead to a faster convergence speed.
\end{example}
\begin{figure}[htb!]\centering{
\begin{tabular}{cc}
		\includegraphics[scale = .28]{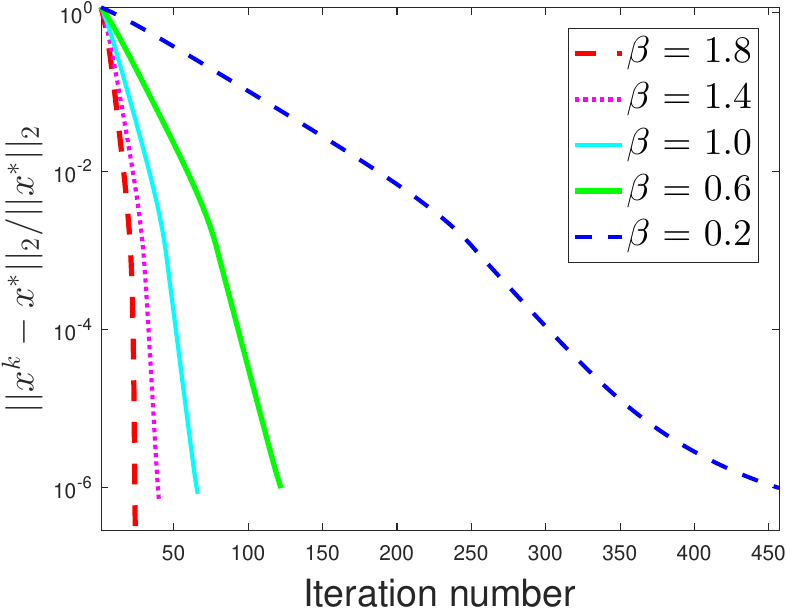} %& \includegraphics[scale = .29]{plots/toyxroutr.eps}
& \includegraphics[scale = .28]{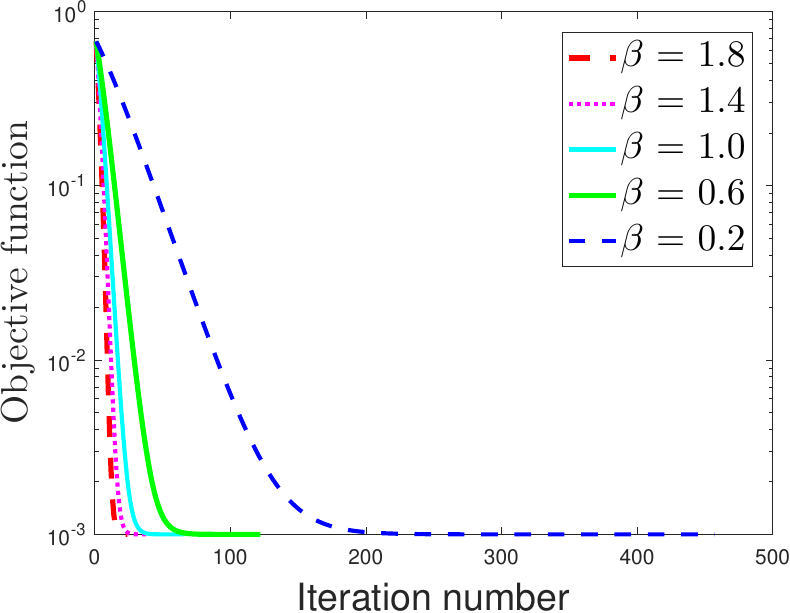}
		\end{tabular}
	}\vspace{-0.2cm} \caption{The evolution of the relative error (left) %the primal sequence $\{{\h x}^k\}$ (middle)
  and the objective function value (right)}.
	\label{RelMSE}\end{figure}

\subsection{Adaptive FSPS algorithm with extrapolation}\label{adaptive}
We present an adaptive version of FSPS, called the Adaptive FSPS algorithm, which determines the parameter sequences $\{\gamma_k\}$ and $\{\delta_k\}$ in a self-adapting manner and ensures the positivity of the sequence $\{\theta_k\}$.

{\small \begin{algo}[Adaptive FSPS algorithm]\label{subx:admm:Bes}
Let $0<\beta<2$, $\nu>0$, $0<q<1$,  $\delta_0,\theta_0>0$, $\gamma_0=1$ and $\varepsilon>0$, and given a starting point $({\h x}^0, {\h z}^0, {\h u}^0)$. For all $k \geq 0$, consider the following update rule:
\begin{align*}
& \text{Choose} \ {\h y}^{k+1}  \in\partial f(K{\h x}^k).\\
& \text{Update} \ {\h x}^{k+1} := \proj\nolimits_{\cal S}\left({\h u}^k+\frac{\theta_k}{\delta_k}K^* {\h y}^{k+1}-\frac{1}{\delta_k}\nabla h({\h x}^k)-\frac{1}{\delta_k}A^* {\h z}^k\right).\\
& \text{Update} \ {\h u}^{k+1} :=(1-\beta) {\h u}^k + \beta {\h x}^{k+1}.\\
& \text{Set} \ \gamma_{k,0}:=\gamma_k.\\
& \text{Find the smallest} \ j_{{k}}\in\{0,1,2,\ldots\} \ \text{such that for} \ \gamma_{k,j_k} :=\gamma_{k,0} q^{j_{k}} \ \text{and} \ \\
& {\h{z}}^{k+1,j_{k}}  :=\prox_{g^*/\gamma_{k,j_k}}\left(\frac{A{\h x}^{k+1}}{\gamma_{k,j_k}}\right)\\
& \text{it holds} \ \theta_{k+1}:=\frac{{\rm\Psi}({\h x}^{k+1},{\h z}^{k+1,{j_{k}}},{\h u}^{k+1}, \delta_k,\gamma_{k,j_k})}{f(K{\h x}^{k+1})}>0.\\
& \text{Update} \ \gamma_{k+1} :=\gamma_{k,j_k}.\\
& \text{Update} \ \delta_{k+1} := 2\nu +L_{\nabla h}+ \frac{2\|A\|^2}{\gamma_{k+1}}.\\
& \text{Update} \ {\h z}^{k+1}:={\h z}^{k+1,j_k}.\\
& \text{If} \  \|{\h z}^{k+1}\|>\min \left (\frac{\varepsilon}{\gamma_{k+1}},\sqrt{\frac{2\varepsilon}{\gamma_{k+1}}}\right ), \ \text{then}\\
			&\;\;\;\;\text{Update}\ \gamma_{k+1} :=\gamma_{k+1} q.\\&\;\;\;\;\text{Update}\ \delta_{k+1} := 2\nu +L_{\nabla h}+ \frac{2\|A\|^2}{\gamma_{k+1}}.\\
%&\text{else}\\
%  & \;\;\;\;\text{Update}\  \gamma_{k+1}=\gamma_k,\  \delta_{k+1}=\delta_k.\\
&\text{End If}
\end{align*}
\end{algo}}
The following assumption ensures the well-posedness of the algorithm.

\begin{ass} \label{ass2} It holds that $A({\cal S}) \subseteq \dom (\partial g)$ and
 there exists a constant $\ell>0$ such that
$
{\rm dist}(0, \partial g(A \h x)) \le \ell \mbox{ for all } \h x  \in {\cal S}.
$
\end{ass}

Note that this condition is automatically satisfied when the compact set $A(\mathcal{S})$ is a subset of the interior of $\dom g$.

\begin{lemma}[Well-definedness of Algorithm \ref{subx:admm:Bes}]\label{lem6} Suppose Assumptions \ref{ass1} and \ref{ass2} hold.  Then the following statements are true:
\begin{itemize}
\item[{\rm (i)}]
 It holds
 $\|{\h z}^{k+1}\| \le \ell + 1$ for all $k \geq 0$.
 \item[{\rm (ii)}]
The procedure of finding the smallest $j_k \in\{0,1,2,\ldots\}$ such that $\theta_{k+1}>0$ is executed in every iteration of Algorithm \ref{subx:admm:Bes} a finite number of times, and so the algorithm is well-defined. Moreover, $\gamma_{k+1}\le \gamma_k$ for all $k \geq 0$.
 \item[{\rm (iii)}] There exists a constant $\gamma >0$ and an index $K_0 \geq 0$ such that $\gamma_k = \gamma>0$,
$\delta_k = \delta:=2\nu +L_{\nabla h}+\frac{2\|A\|^2}{\gamma}$,  and
$\|{\h z}^{k+1}\|\le\min \left (\frac{\varepsilon}{\gamma},\sqrt{\frac{2\varepsilon}{\gamma}}\right)$  for all $k\ge K_0$.
\end{itemize}
 \end{lemma}

  \begin{proof}
  (i) From the construction of the algorithm, $\h x^k \in \mathcal{S}$ for all $k \geq 0$. So, by Assumption \ref{ass2}, for all $k \geq 0$ there exists
  ${\tilde {\h z}}^k\in \partial g(A{\h x}^{k+1})$ with $\|{\tilde {\h z}}^k\| \le \ell +1$.

Taking into account the definitions of ${\h z}^{k+1,j_k}$ and the proximal operator, for all $k \geq 0$, we have
\begin{eqnarray} \label{specialv} && \ g^*({\h z}^{k+1,j_k})- \langle A{\h x}^{k+1},{\h z}^{k+1,j_k}\rangle+\frac{\gamma_{k,j_k}}{2}\|{\h z}^{k+1,j_k}\|^2\nn\\
\le && \ g^*({\tilde {\h z}^k})- \langle A{\h x}^{k+1},{\tilde {\h z}^k}\rangle+\frac{\gamma_{k,j_k}}{2}\|{\tilde {\h z}}^{k}\|^2 =-g(A{\h x}^{k+1})+\frac{\gamma_{k,j_k}}{2}\|{\tilde {\h z}}^{k}\|^2\\
\le && \ g^*({\h z}^{k+1,j_k})- \langle A{\h x}^{k+1},{\h z}^{k+1,j_k}\rangle + \frac{\gamma_{k,j_k}}{2}\|{\tilde {\h z}}^{k}\|^2.\nn
\end{eqnarray}
Since $\gamma_{k,j_k}>0$, it follows that
 $\|{\h z}^{k+1,j_k}\|^2\le \|{\tilde {\h z}}^{k}\|^2 \le (\ell+1)^2,$
consequently, $\|{\h z}^{k+1}\| \le \ell +1$ for all $k \geq 0$.

(ii) Let $k \geq 0$ and $j_k \in \{0, 1, 2, ...\}$. From  \eqref{specialv} and Assumption \ref{ass1}(e), it holds
\begin{align*}
 & \ {\rm\Psi}({\h x}^{k+1},{\h z}^{k+1,{j_{k}}},{\h u}^{k+1}, \delta_k,\gamma_{k,j_k}) \\
= & \ \langle A{\h x}^{k+1}, {\h z}^{k+1,j_{k}}\rangle-g^*({\h z}^{k+1,j_{k}})+h ({\h x}^{k+1})+\frac{\delta_{k}}{2}\|{\h x}^{k+1}-{\h u}^{k+1}\|^2 -\frac{\gamma_{k,j_{k}}}{2}\|{\h z}^{k+1,j_{k}}\|^2 \\
 \geq & \ g(A{\h x}^{k+1})+h({\h x}^{k+1}) +\frac{\delta_{k}}{2}\|{\h x}^{k+1}-{\h u}^{k+1}\|^2-\frac{\gamma_{k,j_{k}}}{2}\|{\tilde {\h z}}^{k}\|^2
\geq \alpha - \frac{\gamma_{k,j_{k}}}{2}\|{\tilde {\h z}}^{k}\|^2.  \end{align*}
Since  $\|{\tilde {\h z}}^{k}\| \le \ell+1$, it is evident that after finitely many increases of $j_k$ with $1$ we obtain $\frac{\gamma_{k,j_k}}{2}\|{\tilde{\h z}}^k\|^2 < \frac{\alpha}{2}$ and, therefore, ${\rm\Psi}({\h x}^{k+1},{\h z}^{k+1,{j_{k}}},{\h u}^{k+1}, \delta_k,\gamma_{k,j_k}) > 0$. Consequently, Algorithm \ref{subx:admm:Bes} is well-defined. Finally, from the formulation of the algorithm we easily see that $\gamma_{k+1} \le \gamma_k$ for all $k \geq 0$.

 (iii) In order to prove the statement, it is sufficient to show that there exist $\gamma >0$ and $K_0 \geq 0$ such that
$\gamma_k = \gamma>0$ for all $k \ge K_0$. Assuming the contrary, there exists a strictly decreasing subsequence $\{\gamma_{k_s}\}$ such that $\gamma_{k_s}\rightarrow 0$
as $s\to+\infty$. As $\|{\h z}^{k_s}\| \le \ell +1$ for all $s \geq 0$,  there exists $s_0 \geq 0$ such that inequality in the ``If-End If'' statement is not verified for all $s \geq s_0$. Therefore, as $\{\gamma_{k_s}\}$ is strictly decreasing, for all $s \geq s_0$ there exists $\widehat{k}_s \in \mathbb{N}$ with $k_s \leq \widehat{k}_s \leq k_{s+1}$ such that $\theta_{\widehat{k}_s + 1,0} \le 0$. Using a similar argument as in {\rm (ii)}, this implies that
$\gamma_{\widehat{k}_s}=\gamma_{\widehat{k}_s,0}\ge\frac{2\alpha}{(\ell + 1)^2}>0$ for all $s \geq s_0$. The monotonicity of the sequence $\{\gamma_{k}\}$ leads to $\gamma_{\widehat{k}_s}\to 0$ as $s \rightarrow +\infty$, and further to a contradiction.
 \end{proof}

\begin{remark}\label{remass}
In the illustrative examples of (a) and (b) provided in Section \ref{sec1}, the functions $f$ and $g$ have a full domain; therefore Assumptions \ref{ass1}(a)-(d) are fulfilled. For example (a), it also holds $\alpha:=\inf_{x \in {\cal S}}g(A{\h x})+h({\h x})>0$ owing the assumption ${\cal B} \cap {\text{span}}({\bf E}) =\emptyset$. In example (b), we have $\alpha:=\inf_{x \in {\cal S}}g(A{\h x})+h({\h x}) \ge 0$; one could then augment the objective with a positive constant as described in Remark \ref{trick}.
\end{remark}

 \begin{remark}\label{remark:4.8} In the special case $\beta=1$, the adaptive FSPS algorithm can alternatively be interpreted using the Moreau envelope of $g$. In fact, in this case, similar to the discussions in Remark \ref{remark:4.2} and by Lemma \ref{lem6}~{\rm (iii)}, for all large $k$, one has $\theta_{k}=\frac{g_{\gamma}(A\h x^{k+1})+h(\h x^{k+1})}{f(K \h x^{k+1})}$, where $\gamma>0$ was given in Lemma \ref{lem6}. Thus, the update of $\h x^{k+1}$ can be interpreted as a proximal subgradient method applied to $F_{\gamma}(x):=\frac{g_{\gamma}(A\h x)+h(\h x)}{f(K \h x)}$ with a fixed but unknown constant $\gamma$. Interestingly, this interpretation does not extend to the general case of $\beta \neq 1$.
\end{remark}

\section{Convergence analysis of the Adaptive FSPS algorithm}\label{sec5}
We provide the convergence analysis for Algorithm \ref{subx:admm:Bes}. To simplify the presentation, we denote
${\bm W}^k:=({\h x}^k,{\h y}^k,{\h z}^k,{\h u}^k)$ for  $\forall k \geq 0.$

We first demonstrate that the iterative sequence generated by the proposed algorithm is bounded. Additionally, we show that any cluster point of the sequence is a limiting approximate stationary point of the underlying fractional program.

\begin{lemma}\label{cofpositive}
Suppose Assumption \ref{ass1} holds. Then, there exist two positive scalars $m$ and $M$ such that
$m< f(K{\h x}^k)\le M$ for all $k \geq 0$.
\end{lemma}

\begin{proof} First, $\sup_{{\h x}\in\cal S} f(K{\h x})<+\infty$ holds due to \cite[Theorem 10.4]{Rock70} and Assumption \ref{ass1}(d). Further, we have $\inf_{k\geq 0} f(K{\h x}^k)>0$. Supposing the contrary, there exists a subsequence $\{{\h x}^{k_j}\}$ such that ${\h x}^{k_j}\to \hat{\h x} \in S$ and $f(K{\h x}^{k_j})\to 0$ as $j\to+\infty$. Using the lower semicontinuity of $f$, this yields $f(K\hat{\h x}) \leq 0$ contradicting Assumption \ref{ass1}(d).
\end{proof}

As shown in Lemma \ref{lem6}, there exists $K_0 \geq 0$ such that $\gamma_k = \gamma$ and $\delta_k = \delta= 2\nu +L_{\nabla h}+\frac{2\|A\|^2}{\gamma}$ for all  $k\ge K_0$. Algorithm \ref{subx:admm:Bes} guarantees $\theta_k>0$, therefore,
${\Psi}({\h x}^{k},{\h z}^{k},{\h u}^{k}, \delta, \gamma)> 0$ for all $k \geq K_0+1$.

\begin{theorem}\label{PriTheo}
Suppose Assumption \ref{ass1} and Assumption \ref{ass2}  hold. Let $0<\beta<2$, $\gamma >0$,  and $K_0 \geq 0$ satisfy $\gamma_k = \gamma>0$,
$\delta_k = \delta:=2\nu +L_{\nabla h}+\frac{2\|A\|^2}{\gamma}$,  and
$\|{\h z}^{k+1}\|\le\min \left (\frac{\varepsilon}{\gamma},\sqrt{\frac{2\varepsilon}{\gamma}}\right)$  for $k\ge K_0$, as indicated by Lemma  \ref{lem6} (iii). Let
$$c_{1}:=\nu,\quad c_{2}:=\delta(2-\beta)/2\beta,\quad c_3:=\gamma/2.$$
Then, for all $k \geq K_0+1$, the following statements are true:
\vspace{0.1cm}
\begin{itemize}
\item[(i)]
${\Psi}({\h x}^{k+1},{\h z}^{k+1},{\h u}^{k+1}, \delta, \gamma)+\theta_k f(K{\h x}^k)-\theta_k\left( \langle K{\h x}^{k+1},{\h y}^{k+1} \rangle-f^*({\h y}^{k+1})\right)\\$
$\le \  {\Psi}({\h x}^{k},{\h z}^{k},{\h u}^{k}, \delta, \gamma)
-c_{1}\|{\h x}^k-{\h x}^{k+1}\|^2-c_{2}\|{\h u}^k-{\h u}^{k+1}\|^2-c_3\|{\h z}^k-{\h z}^{k+1}\|^2;$
\vspace{0.1cm}
 \item[(ii)]
 $ {\Psi}({\h x}^{k+1},{\h z}^{k+1},{\h u}^{k+1}, \delta, \gamma)-\theta_k f(K{\h x}^{k+1})\nn\\
\le
-c_1\|{\h x}^k-{\h x}^{k+1}\|^2-c_2\|{\h u}^k-{\h u}^{k+1}\|^2
-c_3\|{\h z}^{k}-{\h z}^{k+1}\|^2.$
 \end{itemize}
\end{theorem}

\begin{proof} Let $k \geq K_0+1$.
(i) According to the properties of the projection, we have
 \begin{eqnarray*} \h x^{k+1} = \arg\min_{{\h x}\in\cal S}\left[\langle{\h z}^k, A{\h x} \rangle-\theta_k \langle K{\h x},{\h y}^{k+1}\rangle+\langle\nabla h({\h x}^k),{\h x}-{\h x}^k\rangle +\frac{\delta}{2}\|{\h x}-{\h u}^k\|^2\right].\end{eqnarray*}
The objective function of the above optimization problem is strongly convex with modulus $\delta$, therefore,
\begin{align*}
 & \ \langle{\h z}^k, A{\h x}^{k+1}\rangle+\langle {\h x}^{k+1}-{\h x}^k,\nabla h({\h x}^k)\rangle-\theta_k\langle K^*{\h y}^{k+1},{\h x}^{k+1} \rangle
+\frac{\delta}{2}\|{\h x}^{k+1}-{\h u}^k\|^2 \nonumber \\
\le & \ \langle{\h z}^k, A{\h x}^k\rangle-\theta_k\langle K^*
{\h y}^{k+1},{\h x}^k \rangle+\frac{\delta}{2}\|{\h x}^k-{\h u}^k\|^2-\frac{\delta}{2}\|{\h x}^k-{\h x}^{k+1}\|^2.
 \end{align*}
By combining the above inequality with the fact that
$$\langle {\h y}^{k+1},K({\h x}^k-{\h x}^{k+1})\rangle =f(K{\h x}^k)-\left(\langle K{\h x}^{k+1},{\h y}^{k+1} \rangle-f^*({\h y}^{k+1})\right),$$
it yields
{\small \begin{align}\label{desx:key2}
&\langle{\h z}^k, A{\h x}^{k+1}\rangle+\langle {\h x}^{k+1}-{\h x}^k,\nabla h({\h x}^k)\rangle+\theta_k\left[f(K{\h x}^k)- (\langle K{\h x}^{k+1},{\h y}^{k+1} \rangle -f^*({\h y}^{k+1}) )\right]
\nonumber\\
& +\frac{\delta}{2}\|{\h x}^{k+1}-{\h u}^k\|^2
\le \langle{\h z}^k, A{\h x}^k\rangle+\frac{\delta}{2}\|{\h x}^k-{\h u}^k\|^2-\frac{\delta}{2}\|{\h x}^k-{\h x}^{k+1}\|^2.
 \end{align}}
 Since $\nabla h$ is Lipschitz continuous with constant $L_{\nabla h}$, it holds that
\begin{eqnarray}\label{desx:key1}h({\h x}^{k+1})-h({\h x}^k)\le \langle {\h x}^{k+1}-{\h x}^k,
\nabla h({\h x}^k)\rangle +\frac{L_{\nabla h}}{2}\|{\h x}^{k+1}- {\h x}^k\|^2. \end{eqnarray}
 By substituting (\ref{desx:key1}) into (\ref{desx:key2}), we obtain
 \begin{eqnarray*} &&\langle{\h z}^k, A{\h x}^{k+1}\rangle +h({\h x}^{k+1})-h({\h x}^k)+\theta_k\left[f(K{\h x}^k)- (\langle K{\h x}^{k+1},{\h y}^{k+1} \rangle -f^*({\h y}^{k+1}) )\right]
\nn\\ &&+\frac{\delta}{2}\|{\h x}^{k+1}-{\h u}^k\|^2
\le \langle{\h z}^k, A{\h x}^k\rangle+\frac{\delta}{2}\|{\h x}^k-{\h u}^k\|^2-\frac{\delta-L_{\nabla h}}{2}\|{\h x}^k-{\h x}^{k+1}\|^2
 \end{eqnarray*}
or, equivalently,
\begin{align}\label{psides1} & \ \Psi({\h x}^{k+1},{\h z}^k,{\h u}^k, \delta, \gamma) +\theta_k\left[f(K{\h x}^k)- (\langle K{\h x}^{k+1},{\h y}^{k+1} \rangle -f^*({\h y}^{k+1}) )\right]\nn\\
 \le & \ \Psi({\h x}^{k},{\h z}^k,{\h u}^k, \delta, \gamma) -\frac{\delta-L_{\nabla h}}{2}\|{\h x}^{k+1}-{\h x}^k\|^2.\end{align}
According to the definition of the proximal operator, we have
$A{\h x}^{k}-\gamma{\h z}^{k}\in \partial g^*({\h z}^{k}).$  Thus,
$ -g^*({\h z}^{k+1})\le - g^*({\h z}^k)- \langle{\h z}^{k+1} -{\h z}^k, A{\h x}^{k}-\gamma{\h z}^{k}\rangle$, which, in combination with the identity
$-\frac{\gamma}{2}\|{\h z}^{k+1}\|^2 =-\frac{\gamma}{2}\|{\h z}^k\|^2-\gamma\langle{\h z}^{k+1}-{\h z}^k,{\h z}^k \rangle-\frac{\gamma}{2}\|{\h z}^{k+1}-{\h z}^k\|^2,$
gives
{\small \begin{align}\label{PriTheo:Eq1}
& \ \langle{\h z}^{k+1}, A{\h x}^{k+1} \rangle -g^*({\h z}^{k+1}) -\frac{\gamma}{2}\|{\h z}^{k+1}\|^2 \nn\\
\le & \ \langle{\h z}^{k}, A{\h x}^{k+1} \rangle -g^*({\h z}^{k}) -\frac{\gamma}{2}\|{\h z}^{k}\|^2  + \langle{\h z}^{k+1}-{\h z}^k,
 A{\h x}^{k+1}-A{\h x}^k \rangle-\frac{\gamma}{2}\|{\h z}^k-{\h z}^{k+1}\|^2\nn\\
 \le & \ \langle{\h z}^{k}, A{\h x}^{k+1} \rangle -g^*({\h z}^{k}) -\frac{\gamma}{2}\|{\h z}^{k}\|^2 +\frac{\|A\|^2}{2\gamma}\|{\h x}^{k+1}-{\h x}^k\|^2, \end{align}}
where the last inequality follows by the Young's inequality. Using $A{\h x}^{k+1}-\gamma{\h z}^{k+1}\in \partial g^*({\h z}^{k+1})$ and $A{\h x}^{k}-\gamma{\h z}^{k}\in \partial g^*({\h z}^{k})$, and the monotonicity of the subdifferential operator, it yields
$$\gamma\|{\h z}^k-{\h z}^{k+1} \|^2 \leq -\langle {\h z}^k-{\h z}^{k+1},A({\h x}^{k+1}-{\h x}^k)\rangle  \le \frac{\gamma}{2}\|{\h z}^k-{\h z}^{k+1}\|^2+\frac{\|A\|^2}{2\gamma}\|{\h x}^k-{\h x}^{k+1}\|^2,$$
and further, in combination with (\ref{PriTheo:Eq1}),
 \begin{align*}
 & \ \langle{\h z}^{k+1}, A{\h x}^{k+1} \rangle -g^*({\h z}^{k+1}) -\frac{\gamma}{2}\|{\h z}^{k+1}\|^2\nn\\
\le & \ \langle{\h z}^{k}, A{\h x}^{k+1} \rangle -g^*({\h z}^{k}) -\frac{\gamma}{2}\|{\h z}^{k}\|^2 +\frac{\|A\|^2}{\gamma}\|{\h x}^{k+1}-{\h x}^k\|^2-\frac{\gamma}{2}\|{\h z}^k-{\h z}^{k+1}\|^2.
\end{align*}
In other words,
\begin{equation}\label{PriTheo:Phiz}
{\Psi}({\h x}^{k+1},{\h z}^{k+1},{\h u}^{k}, \delta, \gamma)\le{\Psi}({\h x}^{k+1},{\h z}^{k},{\h u}^{k}, \delta, \gamma)+\frac{\|A\|^2}{\gamma}\|{\h x}^{k+1}-{\h x}^k\|^2-\frac{\gamma}{2}\|{\h z}^k-{\h z}^{k+1}\|^2.
\end{equation}
Using the extrapolation step, we get
\begin{align*}
\frac{\delta}{2}\|{\h u}^{k+1}-{\h x}^{k+1}\|^2
 =  \frac{\delta}{2}\|{\h u}^k-{\h x}^{k+1}\|^2-\frac{\delta(1-(1-\beta)^2)}{2\beta^2}\|{\h u}^k-{\h u}^{k+1}\|^2,
\end{align*}
which leads to
\begin{eqnarray}\label{PriTheo:Phiu}
{\Psi}({\h x}^{k+1},{\h z}^{k+1},{\h u}^{k+1}, \delta, \gamma)\le{\Psi}({\h x}^{k+1},{\h z}^{k+1},{\h u}^{k}, \delta, \gamma)-\frac{\delta(2-\beta)}{2\beta}\|{\h u}^k-{\h u}^{k+1}\|^2.
\end{eqnarray}
Finally, by adding (\ref{psides1}), (\ref{PriTheo:Phiz}) and \eqref{PriTheo:Phiu},  the assertion follows by using the definition of $\delta$.

(ii) Follows from (i) by using that $f(K{\h x}^{k+1}) \geq \langle K{\h x}^{k+1},{\h y}^{k+1} \rangle-f^*({\h y}^{k+1})$,
 $\theta_k >0$, and $\theta_k f(K{\h x}^k)={\Psi}({\h x}^{k},{\h z}^{k},{\h u}^{k}, \delta, \gamma)$.
\end{proof}

Let $\gamma$ and $\delta$ be the constants indicated in  Lemma \ref{lem6} (iii), and the merit function $\Pi: {\mathbb R}^n \times {\text {\rm dom}}g^* \times {\mathbb R^n}\rightarrow \overline{{\mathbb \R}}$ defined by
\begin{align*} \Pi ({\h x}, {\h z},{\h u})
= \frac{{\Psi}({\h x},{\h z},{\h u}, \delta,\gamma)}{f(K{\h x})}
= \frac{\langle{\h z},A{\h x} \rangle-g^*({\h z})+h({\h x})+\iota_{\cal S}({\h x}) + \frac{\delta}{2}\|{\h x}-{\h u}\|^2 \!-\! \frac{\gamma}{2}\|{\h z}\|^2}{f(K{\h x})}. \end{align*}

\begin{theorem}[Subsequential convergence]\label{subsequentialtoStat}
Suppose Assumption \ref{ass1} and Assumption \ref{ass2}  hold. Let  $0<\beta<2$, $\gamma >0$  and $K_0 \geq 0$ satisfy $\gamma_k = \gamma>0$,
$\delta_k = \delta:=2\nu +L_{\nabla h}+\frac{2\|A\|^2}{\gamma}$,  and
$\|{\h z}^{k+1}\|\le\min \left (\frac{\varepsilon}{\gamma},\sqrt{\frac{2\varepsilon}{\gamma}}\right)$  for $k\ge K_0$, as indicated by Lemma  \ref{lem6} (iii).  Let
  $\Omega$ be the set of the accumulation points of the sequence $\{{\bm W}^k\}$.  Then, the following statements are true:
\begin{itemize}
\item[(i)]  The sequence $\left \{\!\theta_k \!= \!\frac{{\Psi}({\h x}^k,{\h z}^k,{\h u}^k, \delta,\gamma)}{f(K{\h x}^k)} =  \Pi ({\h x}^k, {\h z}^k,{\h u}^k) \right \}$ is nonincreasing and there exists a  scalar $\overline{\theta}\ge 0$ such that $\lim_{k\to+\infty}\theta_k=\overline{\theta}$.
\item[(ii)] The sequence $\{{\bm W}^k\}$ is bounded.
    \item[(iii)] For every $(\overline{\h x},{\overline{\h y}},{\overline{\h z}},{\overline{\h u}}) \in \Omega$ it holds ${\rm\Pi}(\overline{\h x}, {\overline{\h z}},{\overline{\h u}})={\overline{ \theta}}$.
\item[(iv)]  If $K_{\varepsilon}:=\{{\h x} \, |\, {\rm{dist}}({\h x},A({\cal S}))\le \varepsilon\} \subseteq {\rm int}({\rm dom} g)$, then $g$ is nonsmooth and Lipschitz continuous on the compact set $K_{\varepsilon}$ with some Lipschitz constant $\kappa >0$. In this case, any accumulation point of the sequence $\{{\h x}^k\}$ is a
  limiting $\left(2\kappa\varepsilon,(2\kappa+1)\varepsilon\right)$-lifted approximate stationary point of (\ref{PForm}).

\item[(v)] It holds
\begin{eqnarray}\label{eqde} \lim_{k\to+\infty}\frac{\langle K{\h x}^{k+1}, {\h y}^{k+1} \rangle -f^*({\h y}^{k+1})}{f(K{\h x}^{k})}=1.\end{eqnarray}
Furthermore, there exists an index $K_1\ge K_0+1$ such that
\begin{equation}\label{K1}
0<m\le \langle K{\h x}^{k}, {\h y}^{k} \rangle -f^*({\h y}^{k}) \le f(K{\h x}^k)\le M \quad \forall k \geq K_1,
\end{equation}
where $m$ and $M$ are the bounds from Lemma \ref{cofpositive}.
\end{itemize}
\end{theorem}
\begin{proof} (i) It follows from Theorem \ref{PriTheo} (ii) that for all $k \geq K_0+1$
\begin{align}\label{keydestheta}
\theta_{k+1}\le & \ \theta_k - \frac{1}{f(K{\h x^{k+1}})}\left(c_1 \|{\h x}^k-{\h x}^{k+1}\|^2 + c_2
\|{\h u}^k-{\h u}^{k+1}\|^2  + c_3 \|{\h z}^{k}-{\h z}^{k+1}\|^2 \right)\nonumber\\
\le & \ \theta_k - \frac{1}{M} \left(c_1 \|{\h x}^k-{\h x}^{k+1}\|^2 + c_2
\|{\h u}^k-{\h u}^{k+1}\|^2  + c_3 \|{\h z}^{k}-{\h z}^{k+1}\|^2 \right),
\end{align}
where $M >0$ is the constant provided by Lemma \ref{cofpositive}. Thus,
\begin{eqnarray}\label{acc1}\;\; \|{\h x}^k-{\h x}^{k+1}\|\rightarrow 0,\;\|{\h u}^k-{\h u}^{k+1}\|\rightarrow 0,\;\|{\h z}^{k+1}-{\h z}^k\|\rightarrow 0,\; \|{\h x}^{k+1}-{\h u}^k\|\rightarrow 0,\end{eqnarray}
as $k\to +\infty$ and the sequence $\{\theta_k\}$ is nonincreasing. Thus, ${\overline{\theta}} := \lim_{k\to\infty}\theta_k \geq 0$ exists.

(ii) Since $\cal S$ is a compact set, the sequence $\{{\h x}^k\}$ is bounded  by construction, which, according to (\ref{acc1}), guarantees that $\{{\h u}^k\}$ is bounded. The sequence
 $\{{\h y}^k\}$ is bounded due to Assumption \ref{ass1}(d), and the sequence $\{{\h z}^k\}$ is  bounded due to Assumption \ref{ass2}.

(iii) Let  $\overline{\bm W}=(\overline{\h x},{\overline{\h y}},{\overline{\h z}},{\overline{\h u}})$ be an accumulation point of the sequence and $\{{\bm W}^{k}\}$ and $\{{\bm W}^{k_j}\}$ be a subsequence such that
$\lim_{j\to+\infty} {\bm W}^{k_j}= {\overline {\bm W}}.$

From
$
\lim_{j\to+\infty} \frac{\Psi({\h x}^{k_j},{\h z}^{k_j},{\h u}^{k_j}, \delta, \gamma)}{f(K{\h x}^{k_j})}=\lim_{j\to+\infty}\theta_{k_j} = {\overline{\theta}}
$
and
$\lim_{j\to+\infty}f(K{\h x}^{k_j})=f(K{\overline{\h x}}) >0$, which holds due to
Assumption \ref{ass1}(d), by noting that $\{K{\h x}^{k_j}\}\subseteq K(\mathcal{S}) \subseteq {\rm int} ({\rm dom} f$) and $K(\mathcal{S})$ is closed, we have that the following limit exists:
\begin{eqnarray}\label{Psilim} \overline{\Psi}:=\lim_{j\to\infty}\Psi({\h x}^{k_j},{\h z}^{k_j},{\h u}^{k_j}, \delta, \gamma) \in \mathbb{R}.\end{eqnarray}
\noindent Next, we show that $\overline{\Psi}=\Psi({\overline{\h x}},{\overline{\h z}},{\overline{\h u}}, \delta, \gamma)$. From (\ref{Psilim}), ${\h x}^{k_j} \in \mathcal{S}$, $g^*$ is lower semicontinuous and  the definition of ${\Psi}(\cdot, \cdot, \cdot, \delta, \gamma)$, we have that $\Psi({\overline{\h x}},{\overline{\h z}},{\overline{\h u}}, \delta, \gamma) \ge \overline{\Psi}$.
Invoking the update scheme, for every $j \geq 0$ such that $k_j \geq K_0+1$ it holds
$$g^*({\overline{\h z}})-\langle \overline{\h z},A{\h x}^{k_j} \rangle+\frac{\gamma}{2}\|{\overline{\h z}}\|^2 \ge g^*({\h z}^{k_j})-\langle{\h z}^{k_j},A{\h x}^{k_j} \rangle+\frac{\gamma}{2}\|{\h z}^{k_j}\|^2$$
and, further,
\begin{eqnarray*}
-g^*({\overline{\h z}})+\langle \overline{\h z},A{\h x}^{k_j} \rangle-\frac{\gamma}{2}\|{\overline{\h z}}\|^2 +h({\h x}^{k_j})\le -g^*({\h z}^{k_j})+\langle{\h z}^{k_j},A{\h x}^{k_j} \rangle-\frac{\gamma}{2}\|{\h z}^{k_j}\|^2+h({\h x}^{k_j}).\end{eqnarray*}
We let $j\to+\infty$ and get
\begin{eqnarray*}
-g^*({\overline{\h z}})+\langle \overline{\h z},A{\overline{\h x}} \rangle-\frac{\gamma}{2}\|{\overline{\h z}}\|^2 +h({\overline{\h x}})\le \lim_{j\to+\infty}( -g^*({\h z}^{k_j})+\langle{\h z}^{k_j},A{\h x}^{k_j} \rangle-\frac{\gamma}{2}\|{\h z}^{k_j}\|^2+h({\h x}^{k_j})),
\end{eqnarray*}
so, $\Psi({\overline{\h x}},{\overline{\h z}},{\overline{\h u}}, \delta, \gamma)\le \overline{\Psi}$. In conclusion,  $\Psi({\overline{\h x}},{\overline{\h z}},{\overline{\h u}}, \delta, \gamma) = \overline{\Psi}$ and $\Pi({\overline{\h x}},{\overline{\h z}},{\overline{\h u}})={\overline{\theta}}$.

(iv) Invoking the update rules for ${\h x}^{k+1}$, ${\h y}^{k+1}$, ${\h z}^{k+1}$ and
${\h u}^{k+1}$, for all $k\ge K_0+1$ it yields
\begin{eqnarray}\label{systemk}\left\{\begin{array}{l}{\h y}^{k+1}\in \partial f(K{\h x}^k),\\[0.1cm]
0\in \partial \iota _{\cal S}({\h x}^{k+1})+ A^* {\h z}^k+\nabla h({\h x}^k)-\theta_k K^* {\h y}^{k+1}  +\delta ({\h x}^{k+1}-{\h u}^k),\\[0.1cm]
A{\h x}^{k+1} -\gamma {\h z}^{k+1}\in\partial g^*({\h z}^{k+1}),\\[0.1cm]
{\h u}^{k+1}=(1-\beta) {\h u}^k + \beta {\h x}^{k+1}.\end{array}\right.
\end{eqnarray}
Let $\overline{\bm W}=(\overline{\h x},{\overline{\h y}},{\overline{\h z}},{\overline{\h u}})$ be an accumulation point of the sequence of $\{{\bm W}^{k}\}$,
and let $\{{\bm W}^{k_j}=({\h x}^{k_j},{{\h y}}^{k_j},{{\h z}}^{k_j},{{\h u}}^{k_j})\}$ be a subsequence converging to $\overline{\bm W}$ as $j \rightarrow +\infty$. From (\ref{acc1}), we see that ${\h x}^{k_j-1} \rightarrow \overline{\h x}$ and ${\h u}^{k_j-1} \rightarrow \overline{\h u}$ as $j \rightarrow +\infty$.
Then, letting $k=k_j-1$ and $j \rightarrow +\infty$ in the above system and taking into account the fact that the graph of the convex subdifferential is closed, we obtain
\begin{eqnarray}\label{KKT}\left\{\begin{array}{l}{\overline{\h y}}\in \partial f(K{\overline{\h x}}),\\[0.1cm]
{\bf 0}\in \partial \iota _{\cal S}({\overline{\h x}})+ A^* {\overline{\h z}}+\nabla h({\overline{\h x}})-{\overline\theta} K^* {\overline{\h y}},\\[0.1cm]
A{\overline{\h x}}-\gamma{\overline{\h z}}\in\partial g^*(\overline{\h z}) ,\\[0.1cm]
{\overline{\h u}}={\overline{\h x}}.\end{array}\right.
\end{eqnarray}
Since $\|{\h z}^{k+1}\|\le\min \left (\frac{\varepsilon}{\gamma},\sqrt{\frac{2\varepsilon}{\gamma}}\right)$  for all $k\ge K_0$, it yields
$\|{\overline{\h z}}\| \le \min \left (\frac{\varepsilon}{\gamma},\sqrt{\frac{2\varepsilon}{\gamma}}\right)$ .
The third inclusion relation in (\ref{KKT}) guarantees that
$ {\text{dist}}(\partial g^*({\overline{\h z}}), A{\overline{\h x}})\le \varepsilon$. Therefore, according to Lemma \ref{epsappsol}, $\overline{\h z} \in \partial_{2 \kappa \varepsilon} g(A \overline{\h x})$, which, combined with the first two inclusion relations in (\ref{KKT}), leads to
$0 \in \left(A^* \partial_{2 \kappa \varepsilon} g(A \overline{\h x}) + \nabla h({\overline{\h x}}) + \partial \iota _{\cal S}({\overline{\h x}}) \right) f(K{\overline{\h x}}) -{\overline\Psi} K^* \partial f(K{\overline{\h x}}).$

 As seen in the proof of statement (iii), we have
$\overline{\Psi}= \langle A{\overline{\h x}}, {\overline{\h z}} \rangle - g^*({\overline{\h z}})+h({\overline{\h x}}) -\frac{\gamma}{2}\|{\overline{\h z}}\|^2,$ therefore
\begin{align*}
|\overline{\Psi}-(g(A{\overline{\h x}})+h({\overline{\h x}}))|
= & \ (g(A{\overline{\h x}})+h({\overline{\h x}}))-\left (\langle A{\overline{\h x}}, {\overline{\h z}} \rangle - g^*({\overline{\h z}})+h({\overline{\h x}})-\frac{\gamma}{2}\|{\overline{\h z}}\|^2 \right)\\
= & \ g(A{\overline{\h x}}) +  g^*({\overline{\h z}})-\langle A{\overline{\h x}}, {\overline{\h z}} \rangle +\frac{\gamma}{2}\|{\overline{\h z}}\|^2 \leq 2 \kappa \varepsilon + \varepsilon = (2\kappa+1)\varepsilon.\end{align*}
Thus, ${\overline{\h x}}$
 is a limiting $(2 \kappa \varepsilon,(2\kappa+1)\varepsilon)$-lifted approximate stationary point of (\ref{PForm}).
 %The statement in (ivb) follows from (iva) and the second statement of Lemma \ref{epsappsol}.

(v) Invoking the first inclusion relation in (\ref{KKT}) and Lemma \ref{cofpositive}, we obtain for all $k \geq K_0+1$
\begin{align*} \frac{|f(K{\h x}^k)-\left(\langle K{\h x}^{k+1},{\h y}^{k+1} \rangle-f^*({\h y}^{k+1})\right)|}{f(K{\h x}^k)} = & \ \frac{|\langle {\h y}^{k+1},K({\h x}^k-{\h x}^{k+1})\rangle |}{f(K{\h x}^k)}\\
\le & \ \frac{|\langle {\h y}^{k+1},K({\h x}^k-{\h x}^{k+1})\rangle |}{m}.\end{align*}
Using that ${\h x}^k-{\h x}^{k+1}\rightarrow 0$ as $k \rightarrow +\infty$ and the boundedness of $\{\h y^k\}$, we obtain
\begin{eqnarray}\label{fraclim}\lim_{k\to+\infty}\frac{|f(K{\h x}^k)-\left(\langle K{\h x}^{k+1},{\h y}^{k+1}\rangle -f^*({\h y}^{k+1})\right)|}{f(K{\h x}^k)}=0.\end{eqnarray}
The second statement is a direct consequence of \eqref{fraclim}.
\end{proof}
%\begin{remark}{\color{blue} If $g$  is differentiable with $L_{\nabla g}$-Lipschitz continuous gradient ($L_{\nabla g} >0$) over the whole space, then
%any accumulation point of the sequence $\{{\h x}^k\}$ is an exact
%  limiting lifted approximate stationary point of (\ref{PForm}). We will prove it in Section \ref{sect7}.}
%\end{remark}
Let $m >0$ be the scalar introduced in Lemma \ref{cofpositive}, $\gamma$ and $\delta$ the constants indicated in  Lemma \ref{lem6} (iii), and the following modified merit function $\Gamma:\{({\h x},{\h y}) \in {\mathbb R}^n \times \text{\rm{dom}} f^* : \langle  K {\h x}, {\h y}\rangle -f^*({\h y})> m/2 \} \times {\text {\rm dom}}g^* \times {\mathbb R^n}\rightarrow \overline{{\mathbb \R}}$ defined by
$$ \Gamma ({\h x},{\h y},{\h z},{\h u})
\!\!=  \!\!\frac{{\Psi}({\h x},{\h z},{\h u}, \delta,\gamma)}{\langle K{\h x},{\h y}\rangle-f^*({\h y})}
\!\!=\!\!   \frac{\langle{\h z},A{\h x} \rangle-g^*({\h z})+h({\h x})+\iota_{\cal S}({\h x}) + \frac{\delta}{2}\|{\h x}-{\h u}\|^2 \!-\! \frac{\gamma}{2}\|{\h z}\|^2}{\langle K{\h x},{\h y}\rangle-f^*({\h y})}.$$
In the following we show that values of $\Gamma$ along the sequence $({\h x}^{k},{\h y}^{k},{\h z}^{k},{\h u}^{k})$ converge to $\overline \theta$ as $k \rightarrow +\infty$ and that it takes this value at every point of $\Omega$.

%\begin{remark}
%If $ g$ is a smooth and convex function, we can
%use the conceptual algorithmic framework \eqref{FSPSO} with $\gamma_k \equiv 0$ and redefine $\Psi $ by removing the term $-\frac{\gamma}{2} \| \h {z} \|^2$. Then, we can show that the conceptual algorithmic framework \eqref{FSPSO} converges to a limiting lifted stationary point.
%\end{remark}
\begin{theorem}\label{subsequential}
Suppose Assumption \ref{ass1} and Assumption \ref{ass2}  hold. Let $\gamma >0$,  $0<\beta<2$ and $K_0 \geq 0$ satisfy $\gamma_k = \gamma>0$,
$\delta_k = \delta:=2\nu +L_{\nabla h}+\frac{2\|A\|^2}{\gamma}$,  and
$\|{\h z}^{k+1}\|\le\min \left (\frac{\varepsilon}{\gamma},\sqrt{\frac{2\varepsilon}{\gamma}}\right)$  for $k\ge K_0$, as indicated by Lemma  \ref{lem6} (iii), and $K_1 \geq K_0+1$ such that \eqref{K1} holds, as indicated by Theorem \ref{subsequentialtoStat} (v). Then, the following statements are true:
\begin{itemize}
\item[(i)] There exists $c>0$  such that  for all $k\ge K_1$
{\small \begin{align}\label{gammaDesc}
& \ \Gamma ({\h x}^{k+1},{\h y}^{k+1},{\h z}^{k+1},{\h u}^{k+1}) \nn \\
\le &  \ \Gamma ({\h x}^k,{\h y}^k,{\h z}^k,{\h u}^k)
-c\|{\h x}^k-{\h x}^{k+1}\|^2-c\|{\h u}^k-{\h u}^{k+1}\|^2-c\|{\h z}^k -{\h z}^{k+1}\|^2;
\end{align}}
\item[(ii)]
$\lim_{k\to+\infty} \Gamma ({\h x}^{k},{\h y}^{k},{\h z}^{k},{\h u}^{k})$ exists
and it is equal to ${\overline{\theta}} =\lim_{k \rightarrow +\infty} \theta_k$;
\item[(iii)] For every $(\overline{\h x},{\overline{\h y}},{\overline{\h z}},{\overline{\h u}}) \in \Omega$ it holds ${\rm\Gamma}(\overline{\h x}, {\overline{\h y}}, {\overline{\h z}},{\overline{\h u}})={\overline{ \theta}}$.
\end{itemize}
\end{theorem}

\begin{proof}
(i) By using the fact of $\theta_k f(K{\h x}^k)=\Psi({\h x}^k,{\h z}^k,{\h u}^k, \delta, \gamma)$, from Theorem \ref{PriTheo} (i)  we obtain for all $k \geq K_1$
\begin{align*}
 {\Psi}({\h x}^{k+1},{\h z}^{k+1},{\h u}^{k+1}, \delta, \gamma)\leq &
\ \theta_k\left( \langle K{\h x}^{k+1},{\h y}^{k+1} \rangle-f^*({\h y}^{k+1})\right) \nonumber \\
& \ -c_1\|{\h x}^k-{\h x}^{k+1}\|^2-c_2\|{\h u}^k-{\h u}^{k+1}\|^2-c_3\|{\h z}^{k}-{\h z}^{k+1}\|^2.
 \end{align*}
Since $0 < m \leq \langle K{\h x}^{k+1}, {\h y}^{k+1} \rangle -f^*({\h y}^{k+1})\leq M$ for all $k \geq K_1$, it yields
\begin{eqnarray}\label{Theo4:thetak}&&\eta_{k+1}\le \theta_k
-\frac{c_1}{M}\|{\h x}^k-{\h x}^{k+1}\|^2-\frac{c_2}{M}\|{\h u}^k-{\h u}^{k+1}\|^2-\frac{c_3}{M} \|{\h z}^k-{\h z}^{k+1}\|^2,
\end{eqnarray}
where $\eta_{k+1}:= \Gamma ({\h x}^{k+1},{\h y}^{k+1},{\h z}^{k+1},{\h u}^{k+1}) = \frac{\Psi({\h x}^{k+1},{\h z}^{k+1},{\h u}^{k+1}, \delta, \gamma)}{\langle K{\h x}^{k+1}, {\h y}^{k+1} \rangle -f^*({\h y}^{k+1})}$. Then one can choose $c:=\frac{1}{M}\min({c_1}, {c_2}, {c_3})$ and the conclusion follow as
$\theta_k\le \eta_k$ for all $k \geq K_1$.

(ii) Since $\eta_k\ge\theta_k$ for all $k \geq K_1$, the sequence $\{\eta_k\}_{k \geq K_1}$ is bounded from below. In addition, it is non-increasing, which means that $\lim_{k\to+\infty} \eta_k$ exists and $\lim_{k\to+\infty} \eta_k \geq \lim_{k\to+\infty} \theta_k = \overline \theta$.
 On the other hand, invoking (\ref{Theo4:thetak}),
 it yields
 \begin{align*}
 \lim \limits_{k\to+\infty}\eta_{k+1}\le & \ \lim \limits_{k\to+\infty}(\theta_k-c\|{\h x}^k-{\h x}^{k+1}\|^2-c\|{\h u}^k-{\h u}^{k+1}\|^2-c \|{\h z}^k-{\h z}^{k+1}\|^2)\\
 = & \ \lim\limits_{k\to\infty} \theta_k = \overline \theta,
 \end{align*}
which proves the statement.

(iii) For every $(\overline{\h x},{\overline{\h y}},{\overline{\h z}},{\overline{\h u}}) \in \Omega$ it holds ${\overline{\h y}} \in \partial f(K \overline{\h x})$, therefore
$${\rm\Gamma}(\overline{\h x}, {\overline{\h y}}, {\overline{\h z}},{\overline{\h u}}) = \frac{{\Psi}(\overline{\h x},\overline{\h z},\overline {\h u}, \delta,\gamma)}{{\langle K \overline{\h x},\overline{\h y} \rangle -f^*(\overline{\h y})}} = \frac{{\Psi}(\overline{\h x},\overline{\h z},\overline {\h u}, \delta,\gamma)}{f(K \overline {\h x})} = \overline \theta.$$
The conclusion follows.
\end{proof}

\section{Convergence of the full sequence}\label{sec6}
To this end, we will provide two different settings in which we can bound the distance between the origin and the limiting subdifferential of ${\rm\Gamma}$ and ${\rm\Pi}$, respectively, in terms of the residual between two consecutive iterates.
%As we will see later, the assumptions of these two settings either can be enforced without loss of generality or can be automatically satisfied by many applications (such as the motivation examples in the introduction).

The two settings are considered below under the assumptions of Theorem \ref{subsequential}, namely, by supposing that Assumption \ref{ass1} and Assumption \ref{ass2} hold,  $0<\beta<2$, $\gamma >0$ and $K_0 \geq 0$ satisfy $\gamma_k = \gamma>0$,
$\delta_k = \delta:=2\nu +L_{\nabla h}+\frac{2\|A\|^2}{\gamma}$,  and
$\|{\h z}^{k+1}\|\le\min \left (\frac{\varepsilon}{\gamma},\sqrt{\frac{2\varepsilon}{\gamma}}\right)$  for all $k\ge K_0$, as indicated by Lemma  \ref{lem6} (iii), and $K_1 \geq K_0+1$ is such that \eqref{K1} holds, as indicated by Theorem \ref{subsequentialtoStat} (v).
\vspace{-0.1cm}
\subsection{\texorpdfstring{$f^*$}{fstar} satisfies the calm condition over its effective domain and \texorpdfstring{$g$}{g} is essentially strictly convex}\label{subsec61}

The following characterization of the Fr\'echet subdifferential of the merit function $\Gamma$ follows from Lemma \ref{ratioC2} and the exact Fr\'echet subdifferential formula for the sum of a function and a differentiable function.
\begin{lemma} \label{GamFunCal}
Suppose Assumption \ref{ass1} holds. Let $f^*$ satisfy
the calm condition at $\widehat{\h y}\in{\text{\rm{dom}}} f^*$,
$\widehat{\h x}\in{\cal S}$ be such that
   $\langle  K \widehat{\h x}, \widehat{\h y}\rangle -f^*(\widehat{\h y})> m/2$, and
 $g^*$ be differentiable at ${\widehat{\h z}}\in\text{\rm int}({\text {\rm dom}}g^*)$.
 Denote $\alpha_1:=\Psi(\widehat{\h x},\widehat{\h z},\widehat{\h u}, \delta, \gamma)$ and $\alpha_2:=\langle K{\widehat{\h x}},\widehat{\h y}\rangle-f^*(\widehat{\h y})$, and suppose that $\alpha_1>0$. Then, there exist open sets $\mathcal{O}_i$, $i=1,2$, such that  $\langle K{\widehat{\h x}},\widehat{\h y}\rangle-f^*(\widehat{\h y})> m/2$ for all $(\widehat{\h x},\widehat{\h y})\in{\cal O}_1\times {\cal O}_2$, and
 {\small \begin{align*}
  {\hat\partial} \Gamma(\widehat{\h x},\widehat{\h y},\widehat{\h z},\widehat{\h u})
 \!\!= \!\!\left\{ (\pmb{\xi}_{\h x},\pmb{\xi}_{\h y},\pmb{\xi}_{\h z},\pmb{\xi}_{\h u})\left|
 \begin{array}{l}
 {\pmb{\xi}_{\h x}}\in \displaystyle{\frac{\alpha_2(A^*{\widehat{\h z}}+\nabla h({\widehat{\h x}})+\partial \iota_{\cal S}(\widehat{\h x}) + \delta({\widehat{\h x}}-\widehat{\h u}))-\alpha_1 K^*\widehat{\h y}}{(\langle K \widehat{\h x},\widehat{\h y} \rangle -f^*(\widehat{\h y}))^2}}\\[0.cm]
 {\pmb{\xi}_{\h y}}\in \displaystyle{\frac{\alpha_1({{\partial}f^*}(\widehat{\h y})-K\widehat{\h x})}{(\langle K\widehat{\h x},\widehat{\h y} \rangle -f^*(\widehat{\h y}))^2}}\\[0.cm]
 \pmb{\xi}_{\h z} =\displaystyle{\frac{A\widehat{\h x}-\nabla g^*(\widehat{\h z})-\gamma{\widehat{\h z}}}{\langle K \widehat{\h x},\widehat{\h y} \rangle -f^*(\widehat{\h y})}}\\[0.cm]
 {\pmb{\xi}_{\h u}}=  \displaystyle{\frac{\delta(\widehat{\h u}-{\widehat{\h x}})}{\langle K \widehat{\h x},\widehat{\h y} \rangle -f^*(\widehat{\h y})}}
 \end{array}\right.
 \right\}.
 \end{align*}}
 \end{lemma}

\begin{theorem}\label{residual}
Suppose that $f^*$ satisfies the calm condition over its effective domain and $g$ is essentially strictly convex. Then there exists $\zeta>0$  such that for all $k\ge K_1$
\begin{eqnarray*}&&{\text{\rm dist}}(\h 0,\partial\Gamma ({\h x}^{k+1},{\h y}^{k+1},{\h z}^{k+1},{\h u}^{k+1}))
\le \zeta(\|{\h x}^k-{\h x}^{k+1}\|+\|{\h u}^k -{\h u}^{k+1}\|+\|{\h z}^k-{\h z}^{k+1}\|).\end{eqnarray*}
\end{theorem}

\begin{proof} Let $k \geq K_1$ be fixed. It holds
\begin{eqnarray*}{\text{\rm dist}}(\h 0,\partial\Gamma ({\h x}^{k+1},{\h y}^{k+1},{\h z}^{k+1},{\h u}^{k+1}))
\le
{\text{\rm dist}}(\h 0,{\hat\partial}\Gamma ({\h x}^{k+1},{\h y}^{k+1},{\h z}^{k+1},{\h u}^{k+1})).\end{eqnarray*}
Since $g$ is essentially strictly convex, $g^*$ is
essentially smooth \cite[Theorem 26.3]{Rock70}. According to the third inclusion relation in \eqref{systemk}, we have
$ A{\h x}^{k+1} -\gamma {\h z}^{k+1} \in  \partial g^*({\h z}^{k+1})$, which means ${\h z}^{k+1} \in {\rm int} ({\rm dom}g^*)$. In addition,  $m \leq \langle K{\h x}^{k+1}, {\h y}^{k+1} \rangle -f^*({\h y}^{k+1}) \leq M$ and $\Psi({\h x}^{k+1},{\h z}^{k+1},{\h u}^{k+1}, \delta, \gamma)>0$.
Thus, one can make use of the formula provided in Lemma \ref{GamFunCal} to characterize the subdifferential of
 $\Gamma$ at $({\h x}^{k+1},{\h y}^{k+1},{\h z}^{k+1},{\h u}^{k+1})$.
Invoking again \eqref{systemk}, we have $-(A^*{\h z}^k+\nabla h({\h x}^k)-\theta_k K^* {\h y}^{k+1} +\delta({\h x}^{k+1}-{\h u}^k))\in {\hat\partial} \iota _{\cal S}({\h x}^{k+1})$ and ${\h y}^{k+1} \in \partial f(K{\h x}^{k})$ or, equivalently, $K{\h x}^{k} \in \partial f^*({\h y}^{k+1})$, and  $A{\h x}^{k+1} -(\nabla g^*({\h z}^{k+1})+\gamma {\h z}^{k+1})=0$.

Thus, for
 \begin{eqnarray*}
\pmb{\xi}_{\h x}^{k+1} &:= & \!\! \frac{A^*{\h z}^{k+1} +\nabla h({\h x}^{k+1})+\delta ({\h x}^{k+1}-{\h u}^{k+1})}{\langle K{\h x}^{k+1},{\h y}^{k+1}\rangle-f^*({\h y}^{k+1})}
 \!-\!\! \ \frac{A^*{\h z}^k+\nabla h({\h x}^k)-\theta_k K^* {\h y}^{k+1} +\delta({\h x}^{k+1}-{\h u}^k)}{\langle K{\h x}^{k+1},{\h y}^{k+1}\rangle-f^*({\h y}^{k+1})}\\
& -& \ \frac{\Psi({\h x}^{k+1},{\h z}^{k+1},{\h u}^{k+1}, \delta, \gamma) K^*{\h y}^{k+1}}{(\langle K{\h x}^{k+1},{\h y}^{k+1}\rangle-f^*({\h y}^{k+1}))^2}\\
 \pmb{\xi}_{\h y}^{k+1} &:= & \ \frac{\Psi({\h x}^{k+1},{\h z}^{k+1},{\h u}^{k+1}, \delta, \gamma)(-K{\h x}^{k+1}+K{\h x}^{k})}{(\langle K{\h x}^{k+1},{\h y}^{k+1}\rangle-f^*({\h y}^{k+1}))^2},\\
 {\pmb{\xi}_{\h z}^{k+1}}&:= & \ \frac{A{\h x}^{k+1} -(\nabla g^*({\h z}^{k+1})+\gamma {\h z}^{k+1})}{ \langle K{\h x}^{k+1}, {\h y}^{k+1} \rangle -f^*({\h y}^{k+1})} = 0, \; \;\pmb{\xi}_{\h u}^{k+1}:= \frac{\delta({\h u}^{k+1}-{\h x}^{k+1})}{\langle K{\h x}^{k+1},{\h y}^{k+1}\rangle-f^*({\h y}^{k+1})},
 \end{eqnarray*}
we have that $(\pmb{\xi}^{k+1}_{\h x},\pmb{\xi}^{k+1}_{\h y},\pmb{\xi}^{k+1}_{\h z},\pmb{\xi}^{k+1}_{\h u})\in{\hat\partial} \Gamma({\h x}^{k+1},{\h y}^{k+1},{\h z}^{k+1},{\h u}^{k+1}).$ Consequently,
\begin{equation}\label{firstbound}
{\text{\rm dist}}(\h 0,{\hat\partial}\Gamma ({\h x}^{k+1},{\h y}^{k+1},{\h z}^{k+1},{\h u}^{k+1})) \leq \|\pmb{\xi}_{\h x}^{k+1}\| + \|\pmb{\xi}_{\h y}^{k+1}\| + \|\pmb{\xi}_{\h u}^{k+1}\|.
\end{equation}

Due to the boundedness of the four sequences, the values
$$B_{\h x}:=\sup\limits_{k} \|{\h x}^k\|, \ B_{\h y}:=\sup\limits_{k}\|{\h y}^k\|, \ B_{\h z}:=\sup\limits_{k}\|{\h z}^k\|, \ B_{\h u}:=\sup\limits_{k}\|{\h u}^k\|$$
 are finite. Since $\{\theta_k\}$ and $\{f(K{\h x}^{k})\}$ are bounded, the sequence $\{\Psi({\h x}^{k},{\h z}^{k},{\h u}^{k}, \delta, \gamma)\}$ is also bounded. Let $B_{\Psi}:=\sup\limits_{k}|\Psi({\h x}^{k},{\h z}^{k},{\h u}^{k}, \delta, \gamma)|<+\infty$. Further, as  $\{{\h z}^k\} \subseteq \inte(\dom g^*)$, $g^*$ is Lipschitz continuous on the closure of $\{{\h z}^k\}$. We denote by $L_{g^*}$ the corresponding Lipschitz constant.
This being given, it is evident that
  \begin{align}\label{boundPsi} & \ |\Psi({\h x}^{k},{\h z}^{k},{\h u}^{k}, \delta, \gamma) -\Psi({\h x}^{k+1},{\h z}^{k+1},{\h u}^{k+1}, \delta, \gamma) | \nonumber \\
  \le & \ \varrho_1\|{\h x}^k-{\h x}^{k+1}\|
  +\varrho_2 \|{\h z}^k-{\h z}^{k+1}\|
  +\varrho_3\|{\h u}^{k+1}-{\h u}^k \|,\end{align}
  where $\varrho_1:=B_{\h z}\|A\|+\delta(B_{\h x}+B_{\h u})+ L_{h}$, $\varrho_2:=\|A\|B_{\h x}+L_{g^*}+\gamma B_{\h z}$,
  $\varrho_3:=\delta(B_{\h x}+B_{\h u})$. Since
\begin{align*}
\pmb{\xi}_{\h x}^{k+1} = & \ \frac{A^*({\h z}^{k+1} - {\h z}^{k}) +\nabla h({\h x}^{k+1}) - \nabla h({\h x}^{k}) +  \delta ({\h u}^{k}-{\h u}^{k+1})}{\langle K{\h x}^{k+1},{\h y}^{k+1}\rangle-f^*({\h y}^{k+1})}  \\
& + \! \frac{\Psi({\h x}^{k},{\h z}^{k},{\h u}^{k}, \delta, \gamma)  \frac{\langle K{\h x}^{k+1}, {\h y}^{k+1}\rangle-f^*({\h y}^{k+1})}{f(K{\h x}^k)}-\Psi({\h x}^{k+1},{\h z}^{k+1},{\h u}^{k+1}, \delta, \gamma)  }{(\langle K{\h x}^{k+1},{\h y}^{k+1}\rangle-f^*({\h y}^{k+1}))^2} K^*{\h y}^{k+1},
\end{align*}
we obtain
\begin{align*}
\|\pmb{\xi}_{\h x}^{k+1}\| \leq & \ \frac{1}{m}(L_{\nabla h}\|{\h x}^k-{\h x}^{k+1}\| + \|A\|\|{\h z}^{k}-{\h z}^{k+1}\| +\delta\|{\h u}^k-{\h u}^{k+1}\|) \\
& + \frac{B_{\h y}\|K\|}{m^2} \left | \Psi({\h x}^{k},{\h z}^{k},{\h u}^{k}, \delta, \gamma)  \left( \frac{\langle K{\h x}^{k+1}, {\h y}^{k+1}\rangle-f^*({\h y}^{k+1})}{f(K{\h x}^k)}- 1 \right) \right |\\
& + \frac{B_{\h y}\|K\|}{m^2} \left | \Psi({\h x}^{k},{\h z}^{k},{\h u}^{k}, \delta, \gamma) - \Psi({\h x}^{k+1},{\h z}^{k+1},{\h u}^{k+1}, \delta, \gamma)\right |.
\end{align*}
From
\begin{align*}
& \left | \Psi({\h x}^{k},{\h z}^{k},{\h u}^{k}, \delta, \gamma)  \left( \frac{\langle K{\h x}^{k+1}, {\h y}^{k+1}\rangle-f^*({\h y}^{k+1})}{f(K{\h x}^k)}- 1 \right) \right |\\
= & \left | \Psi({\h x}^{k},{\h z}^{k},{\h u}^{k}, \delta, \gamma)  \frac{\langle K({\h x}^{k+1}-{\h x}^{k}),{\h y}^{k+1}\rangle}{f(K{\h x}^k)} \right | \le \frac{B_{\Psi}B_{\h y}\|K\|}{m}\|{\h x}^k-{\h x}^{k+1}\|,
\end{align*}
and \eqref{boundPsi}, it yields
$ \|\pmb{\xi}_{\h x}^{k+1}\|\le \eta_1\|{\h x}^k-{\h x}^{k+1}\|+\eta_2\|{\h u}^k-{\h u}^{k+1}\|+\eta_3\|{\h z}^k-{\h z}^{k+1}\|,$
with $\eta_1:= \frac{L_{\nabla h}}{m} + \frac{B_{\Psi}B_{\h y}^2\|K\|^2}{m^3} + \varrho_1 \frac{B_{\h y}\|K\|}{m^2}$, $\eta_2:=\frac{\delta}{m} +  \varrho_3 \frac{B_{\h y}\|K\|}{m^2}$ and $\eta_3:=\frac{\|A\|}{m} +  \varrho_2 \frac{B_{\h y}\|K\|}{m^2}$.

\noindent In addition, we have that
 $\|\pmb{\xi}_{\h y}^{k+1}\|\le B_{\Psi}\frac{\|K\|}{m^2}\|{\h x}^k-{\h x}^{k+1}\| \ \mbox{and} \
 \|\pmb{\xi}_{\h u}^{k+1}\|\le \frac{\delta|1-\beta|}{m\beta}\|{\h u}^k-{\h u}^{k+1}\|,$
which, in the light of \eqref{firstbound}, leads to the conclusion.
\end{proof}
\vspace{-0.1cm}
\subsection{\texorpdfstring{$f$}{f} is  differentiable with Lipschitz continuous gradient over an open set containing \texorpdfstring{$K({\cal S})$}{KS} and \texorpdfstring{$g$}{g} is essentially strictly convex}\label{subsec62}
The workhorse of our analysis will be the merit function $\Pi$. The following statement is a direct consequence of \cite[Lemma 2.1 (ii)]{BDL}.

\begin{lemma} \label{GamFunCal2}
Suppose Assumption \ref{ass1} holds. Let $f$ be differentiable at $K \widehat{\h x} \in \inte(\dom f)$ for $\widehat{\h x} \in {\cal S}$, and
 $g^*$ be differentiable at ${\widehat{\h z}}\in\text{\rm int}({\text {\rm dom}}g^*)$. Denote $\alpha_1:=\Psi(\widehat{\h x},\widehat{\h z},\widehat{\h u}, \delta, \gamma)$ and $\alpha_2:= f(K\widehat{\h x})$, and suppose that $\alpha_1>0$. Then,
 {\small \begin{align*}
  {\hat\partial} \Pi(\widehat{\h x},\widehat{\h z},\widehat{\h u})
 \! = \!\!\left\{ \!\!(\pmb{\xi}_{\h x},\pmb{\xi}_{\h z},\pmb{\xi}_{\h u})\left|
 \begin{array}{l}
 {\pmb{\xi}_{\h x}}\in \displaystyle{\frac{\alpha_2(A^*{\widehat{\h z}}+\nabla h({\widehat{\h x}})+\partial \iota_{\cal S}(\widehat{\h x}) + \delta({\widehat{\h x}}-\widehat{\h u}))-\alpha_1 K^*\nabla f(K{\widehat{\h x}})}{(f(K{\widehat{\h x}}))^2}}\\[0.cm]
 \pmb{\xi}_{\h z} =\displaystyle{\frac{A\widehat{\h x}-\nabla g^*(\widehat{\h z})-\gamma{\widehat{\h z}}}{f(K{\widehat{\h x}})}}\\[0.cm]
 {\pmb{\xi}_{\h u}}=  \displaystyle{\frac{\delta(\widehat{\h u}-{\widehat{\h x}})}{f(K{\widehat{\h x}})}}
 \end{array}\right.\!\!
 \right\}.
 \end{align*}}
 \end{lemma}

\begin{theorem}\label{residualDiffF}
Suppose that $f$ is differentiable with Lipschitz continuous gradient on an open set containing $K({\cal S})$, and $g$ is   essentially strictly convex. Then there exists  $\zeta>0$ such that for all $k\ge {K}_1$
\begin{eqnarray*}&&{\text{\rm dist}}(\h 0,\partial\Pi ({\h x}^{k+1},{\h z}^{k+1},{\h u}^{k+1}))
\le \zeta(\|{\h x}^k-{\h x}^{k+1}\|+\|{\h u}^k-{\h u}^{k+1}\|+\|{\h z}^k -{\h z}^{k+1}\|).\end{eqnarray*}
\end{theorem}
\begin{proof} The proof is similar to Theorem \ref{residual}, thus omitted here.

 \end{proof}

\begin{remark}[Comments on the assumption of essential strict convexity] \label{constr}
The assumption of $g$ being essentially strictly convex can be enforced without loss of generality, by redefining the functions $g$ and $h$ as ${\widetilde g}({\h x}):=g({\h x})+\frac{s}{2}\|{\h x}\|^2$ and
${\widetilde h}({\h x}):=h({\h x})-\frac{s}{2}\|A{\h x}\|^2$ with $s>0$.
 Thus,  $\tilde g$ is strongly convex, and so, essentially strictly convex.
 In our numerical experiments, we noticed that, for small $s>0$, the algorithm exhibits comparable (or simply the same) numerical performance as for $s=0$.
\end{remark}

\begin{remark}\label{rem66}
We  require that either $f^*$ satisfies the calm condition over its effective domain or $f$ is differentiable with Lipschitz continuous gradient over an open set containing \texorpdfstring{$K({\cal S})$}{KS}. These conditions can be satisfied in many applications. For example, if $f$ is supercocercive, that is, $\lim_{\|\h x\| \rightarrow +\infty} \frac{f(\h x)}{\|\h x\|}=+\infty$, then $f^*$ is a real-valued convex function with full domain \cite[Proposition 14.15]{BC17b}, and so, it is locally Lipschitz (and, in particular, calm). This applies, for instance, to example (b) in the introduction. Regarding example (a), if $p \in (1,+\infty)$, noting that \texorpdfstring{$K({\cal S})$}{KS} is a compact set which does not contain the origin, then  $f=\|\cdot\|_p$ is differentiable with Lipschitz continuous gradient over an open set containing \texorpdfstring{$K({\cal S})$}{KS}.
\end{remark}

\subsection{Global convergence}\label{subsec63}

The proof of the global convergence theorem is in line with the techniques and ideas of \cite[Theorem 4]{LiPong15} and \cite[Theorem 3.4]{BCN19}.

\begin{theorem} \label{theo8}
Let $\varepsilon >0$. Suppose Assumption \ref{ass1} and Assumption \ref{ass2}  hold, $K_{\varepsilon}:=\{{\h x} \ | \ {\text{\rm dist}}({\h x},A({\cal S})) \le \varepsilon\} \subseteq \text{\rm int} (\dom g)$, $g$ is nonsmooth and essentially strictly convex and one of the following conditions are fulfilled:
  \begin{itemize}
  \item[{\rm (i)}] $f^*$ satisfies the calm condition over its effective domain and $\Gamma$ satisfies the KL property at every point of $K_{\varepsilon} \cap {\rm dom} \, \partial (\Gamma)$.
 \item[{\rm (ii)}] $f$ is differentiable with Lipschitz continuous gradient over an open set containing $K({\cal S})$ and $\Pi$ satisfies the KL property at every point of $K_{\varepsilon} \cap {\rm dom} \, \partial (\Pi)$.
     \end{itemize}
Let $\{{\bm W}^k=({\h x}^k,{\h y}^k,{\h z}^k,{\h u}^k)\}$ be the sequence generated by Algorithm \ref{subx:admm:Bes}. Then,
$$
\sum_{k} \Big(\|{\h x}^k-{\h x}^{k+1}\| + \|{\h u}^k-{\h u}^{k+1}\| + \|{\h z}^k -{\h z}^{k+1}\| \Big) <+\infty,
$$
and $\{\h x^k\}$ converges to a limiting $(2\kappa\varepsilon,(2\kappa+1)\varepsilon)$-lifted stationary point of (\ref{PForm}), where $\kappa$ is the Lipschitz constant of $g$ on $K_{\varepsilon}$.
%In addition, if $g$ has full domain and is differentiable with $L_{\nabla g}$-Lipschitz continuous gradient ($L_{\nabla g} >0$), then $\{\h x^k\}$ converges to a limiting $\left(\frac{L_{\nabla g}\varepsilon^2}{2},\frac{L_{\nabla g}\varepsilon^2}{2}+\varepsilon\right)$-lifted stationary point  of (\ref{PForm}).
\end{theorem}
\begin{proof}
We  prove the statement only in the setting of assumption (i). The proof of the other case can be done analogously.
The sequence $\{ \Gamma ({\h x}^{k},{\h y}^{k},{\h z}^{k},{\h u}^{k})\}_{k \geq K_1}$ is nonincreasing and it converges to ${\bar{\theta}}$ as $k \rightarrow +\infty$. Thus, $\Gamma ({\h x}^{k},{\h y}^{k},{\h z}^{k},{\h u}^{k})\ge {\bar{\theta}}$ for all $k \geq K_1$, which allows us to divide the proof into two cases.
Case I. There exists $K_2 \geq K_1$ such that $\Gamma ({\h x}^k,{\h y}^k,{\h z}^k,{\h u}^k) = {\bar{\theta}}$ for  $k\ge K_2$. Then,
$({\h x}^{k+1},{\h z}^{k+1},{\h u}^{k+1}) = ({\h x}^{k},{\h z}^{k},{\h u}^{k})$ for all $k\ge K_2$ due to (\ref{gammaDesc}), and the conclusion follows.
Case II. $ \Gamma ({\h x}^{k},{\h y}^{k},{\h z}^{k},{\h u}^{k})>{\bar{\theta}}$ for all $k \geq K_1$. Let $\Omega $ denote the set of accumulation points of $\{({\h x}^k, {\h y}^k, {\h z}^k, {\h u}^k)\}$. Then, $\Omega$ is compact. Invoking Theorem \ref{subsequentialtoStat} (iii),
according to the uniformized KL property \cite{BST14}, there exist $\varrho>0$ and $\mu>0$ and a desingularization function $\phi$ with the property that for all $(\h x, \h y, \h z, \h u)$ with  ${\text{dist}}((\h x, \h y, \h z, \h u), \Omega)<\varrho$ and ${\bar{\theta}}<\Gamma ({\h x},{\h y},{\h z},{\h u}) <{\bar{\theta}}+\mu$, it holds $\phi'(\Gamma(\h x, \h y, \h z, \h u) - {\bar{\theta}}){\text{\rm{dist}}}(\h 0,\partial \Gamma (\h x, \h y, \h z, \h u)) \ge 1$. Then, there exists $K_2 \geq K_1$ such that
${\text{dist}}(({\h x}^k, {\h y}^k, {\h z}^k, {\h u}^k), \Omega)<\varrho$ and ${\bar{\theta}}<\Gamma ({\h x}^k, {\h y}^k, {\h z}^k, {\h u}^k)<{\bar{\theta}}+\mu$ for all $k \geq K_2$.

Thus, by using Theorem \ref{subsequential} and Theorem \ref{residual}, for all $k \geq K_2$ it holds
\begin{eqnarray*}& \ & \phi(\Gamma({\h x}^k,{\h y}^k,{\h z}^k,{\h u}^k) - {\bar{\theta}}) -\phi(\Gamma({\h x}^{k+1},{\h y}^{k+1},{\h z}^{k+1},{\h u}^{k+1})-{\bar{\theta}})\\
&\ge & \ \phi'(\Gamma ({\h x}^k,{\h y}^k,{\h z}^k,{\h u}^k) - {\bar{\theta}})\left(\Gamma({\h x}^k,{\h y}^k,{\h z}^k,{\h u}^k) - \Gamma({\h x}^{k+1},{\h y}^{k+1},{\h z}^{k+1},{\h u}^{k+1})\right)\\
&\ge & \ \frac{c}{{\text{\rm{dist}}}(\h 0,\partial \Gamma ({\h x}^k,{\h y}^k,{\h z}^k,{\h u}^k))} \left(\|{\h x}^k-{\h x}^{k+1}\|^2+\|{\h u}^k-{\h u}^{k+1}\|^2+\|{\h z}^k-{\h z}^{k+1}\|^2 \right)\\
&\ge &  \frac{c}{3\zeta} \frac{(\|{\h x}^k-{\h x}^{k+1}\|+\|{\h u}^k-{\h u}^{k+1}\|+\|{\h z}^k-{\h z}^{k+1}\|)^2}{(\|{\h x}^k-{\h x}^{k-1}\|+\|{\h u}^k-{\h u}^{k-1}\|+\|{\h z}^k-{\h z}^{k-1}\|)},
\end{eqnarray*}
where $c$ and $\zeta$ are given as in Theorems \ref{subsequential} and \ref{residual}.

 By denoting $\delta_{{\h x},{\h z},{\h u}}^k:=\|{\h x}^k-{\h x}^{k+1}\|+\|{\h u}^k-{\h u}^{k+1}\|+\|{\h z}^k-{\h z}^{k+1}\|$, it follows that for all $k \geq K_2$
\begin{align*} 2\delta_{{\h x},{\h z},{\h u}}^k\le & \  2\sqrt{ \frac{3\zeta}{c}\left(\phi(\Gamma ({\h x}^k,{\h y}^k,{\h z}^k,{\h u}^k) - {\bar{\theta}}) -\phi(\Gamma ({\h x}^{k+1},{\h y}^{k+1},{\h z}^{k+1},{\h u}^{k+1})-{\bar{\theta}}) \right)
 \delta_{{\h x},{\h z},{\h u}}^{k-1}}\nn\\
\le & \ \delta_{{\h x},{\h z},{\h u}}^{k-1}+  \frac{3\zeta}{ c}\left(\phi(\Gamma ({\h x}^k,{\h y}^k,{\h z}^k,{\h u}^k) - {\bar{\theta}}) -\phi(\Gamma ({\h x}^{k+1},{\h y}^{k+1},{\h z}^{k+1},{\h u}^{k+1})-{\bar{\theta}})\right).
\end{align*}
So, the conclusion follows.
\end{proof}
\begin{remark}\label{semialgebraic}
If ${\cal S}$ is a semialgebraic set and $f$, $g$ and $h$ are semialgebraic functions (that is, their graphs can be written as finite union or intersection of sets described by polynomial inequalities), then $\Gamma$ and $\Pi$ are also a semialgebraic functions.  So, they satisfy the KL property at every point of the domain of their subdifferential \cite{ABRS,ABS}.
\end{remark}

\begin{remark}\label{rationale}
As seen in the proof of Theorem \ref{subsequentialtoStat}, every accumulation point $(\overline{\h x}, \overline{\h y}, \overline{\h z}, \overline{\h u})$  of the sequence $({\h x}^k, {\h y}^k, {\h z}^k, {\h u}^k)$ need to fulfill the system of optimality conditions (\ref{KKT}). Due to the existence of $\gamma > 0$ in the third inclusion of (\ref{KKT}), it is not feasible to anticipate $\overline{\h x}$ as an exact limiting lifted stationary point of (\ref{PForm}).

 Since the sequence of $\gamma_k$ is non-increasing, the convergence or divergence of the FSPS algorithm is the same whether with $\gamma_k\equiv0$ or
 $\gamma_k=0$ for $k\ge k_0$.
A question arises regarding whether setting the sequence ${\gamma_k}$ to identically zero in the conceptual algorithm FSPS generates a sequence converging to  exact limiting lifted stationary points. The following example demonstrates that this is not the case.
For the fourth update block  in FSPS when $\gamma_k=0$, we will choose ${\h{z}}^{k+1}$ as the minimum norm solution.

 \begin{example} \label{ex6dot10} Consider the problem (\ref{PForm}) for  ${\cal S}={[0,1]^2} \subseteq \R^2$, $A=K=I$ where $I$ is the identity mapping, and $g,h,f:\mathbb{R}^2 \rightarrow \mathbb{R}$ are given by
$g({\h x}) = \|{\h x}\|_{1}$, $h({\h x}) = \frac{1}{2}\|{\h x}\|^2+{\h e}^\top {\h x}+\frac{1}{2}$ and  $f({\h x}) = {\h e}^\top {\h x}+\frac{1}{2}$, where ${\h e} = (1,1)^\top$. For $\beta=1$, $\gamma_k \equiv 0$, $\delta_k \equiv 1$, $\theta_0:=2$ and ${\h z}^0 = {\h u}^0 = {\h x}^0:= (1,0)^\top$, FSPS generates a sequence $({\h x}^k, {\h y}^k, {\h z}^k, {\h u}^k)$ such that
$${\h z}^k = {\h u}^k = {\h x}^k=\left\{\begin{array}{rl} (0,1)^\top, & \mbox{if} \ k \ \mbox{is odd},\\
(1,0)^\top, & \mbox{if} \ k \ \mbox{is even}, \end{array}
\right. \quad {\h y}^k={\h e} \quad \mbox{and} \quad \theta_k=2 \quad \forall k \geq 1. $$
One can verify that neither $(1,0)^\top$ nor $(0,1)^\top$ is a limiting lifted stationary point of Example \ref{ex6dot10}. Thus,
even the {\it subsequential convergence to an exact limiting lifted stationary point} cannot be guaranteed in this case.
 On the other hand, even if the parameter $\gamma_k\downarrow 0$ with $\gamma_k> 0$ for all $k$, any accumulation point of
the sequence generated by the FSPS may not be a limiting lifted stationary point, we refer to Appendix \ref{counter1} for more detail.
\end{example}
\end{remark}

\section{Further discussions}
\label{sect7}
 As shown in Theorem \ref{theo8}, under the assumptions there, the sequence $\{\h x^k\}$ generated by the Adaptive FSPS (Algorithm \ref{subx:admm:Bes}) converges to an approximate lifted stationary point of \eqref{PForm}.  It is interesting to see when the basic algorithm FSPS can converge to an exact limiting lifted stationary point.
From Example \ref{ex6dot10}, we observe that none of the accumulation points of the sequence generated by the FSPS algorithm (with \(\gamma_k = 0\) for \(k \geq k_0\)) is a limiting lifted stationary point when \(g\) is nonsmooth.

 %The first two questions address
% the relationship between the condition on $g$ and the achieved solution from  the conceptual algorithmic framework FSPS (\ref{FSPSO}).
%The third question regards why we choose  the origin of the
%squared norm term ${\h z}^{k+1}$
% as the regularization term in the updating scheme of ${\h z}^{k+1}$.
%${\h z}$,  is the origin of the squared norm term chosen as the regularization term.

Consider the conceptual algorithm FSPS  with $\gamma_k\equiv 0$ reads for all $k \geq 0$:

\begin{equation}\label{Z-FSPSO}
\vspace{-0.16cm} \left \{
\begin{array}{rcl}
{\h y}^{k+1} & \in & \partial f(K{\h x}^k)\\[0.1cm]
 \h x^{k+1} & = & {\text{Proj}}_{\cal S}\left({\h u}^k+\frac{\tilde\theta_k}{\delta_k}K^* {\h y}^{k+1}-\frac{1}{\delta_k}\nabla h({\h x}^k)-\frac{1}{\delta_k}A^* {\h z}^k\right),\\[0.1cm]
{\h u}^{k+1} & = & (1-\beta) {\h u}^k+\beta {\h x}^{k+1},\\[0.1cm]
{\h{z}}^{k+1} & = &\arg\min_{\h z} \left[ g^*({\h z})- \langle A{\h x}^{k+1}, {\h z} \rangle\right],\\[0.1cm]
\tilde\theta_{k+1} & = & \displaystyle{\frac{\tilde{\rm\Psi}({\h x}^{k+1},{\h z}^{k+1},{\h u}^{k+1};\delta_k)}{f(K{\h x}^{k+1})}}, \vspace{-0.13cm}
\end{array}
\right.
\end{equation}
where
%\begin{eqnarray}\label{tildepsi}
${\tilde\Psi}({\h x},{\h z},{\h u}, \delta)\!:=\!\langle{\h z},A{\h x} \rangle-g^*({\h z})+h({\h x})+\iota_{\cal S}({\h x}) + \frac{\delta}{2}\|{\h x}-{\h u}\|^2$.\footnote{Note that $\gamma_k\equiv 0$. Then, the function $\Psi$ in (\ref{FSPSO}) reduces to $\tilde \Psi$.}
%\end{eqnarray}

Next, we show that the sequence $\{\h x^k\}$ generated by \eqref{Z-FSPSO} converges to an exact limiting lifted stationary point of \eqref{PForm} if $g$ is further assumed to be $\ell$-smooth.
(i.e., $g$ is differentiable and its gradient $\nabla g$ is Lipschitz continuous with constant $\ell > 0$). %Before proceeding, we note that  the 

\begin{theorem} \label{Excttheo8}
Suppose Assumption \ref{ass1} holds,  $g$ is $\ell$-smooth ($\ell >0 $) and essentially strictly convex, and one of the following conditions is fulfilled:
  \begin{itemize}
  \item[{\rm (i)}] $f^*$ satisfies the calm condition over its effective domain and $\Gamma$ satisfies the KL property at every point of ${\rm dom} (\partial  \Gamma)$.
 \item[{\rm (ii)}] $f$ is differentiable with Lipschitz continuous gradient over an open set containing $K({\cal S})$ and $\Pi$ satisfies the KL property at every point of ${\rm dom} (\partial   \Pi)$.
     \end{itemize}
Let $0<\beta<2$,  \!\! $\nu>0$, \!$\delta_k\equiv \delta \!\! :=2\nu + L_{\nabla h} + 2\ell \|A\|^2$  \!\! \!\! for $k \geq 0$, and $\{{\bm W}^k=({\h x}^k,{\h y}^k,{\h z}^k,{\h u}^k)\}$ be the sequence generated by (\ref{Z-FSPSO}).
%(or variable step sizes  satisfying $\nu:=\inf_k\frac{1}{2}(\delta_k - L_{\nabla h} - \ell \|A\|^2)>0$),
Then,
$\sum_{k}\big(\|{\h x}^k-{\h x}^{k+1}\| + \|{\h u}^k-{\h u}^{k+1}\| + \|{\h z}^k -{\h z}^{k+1}\|\big)<+\infty,$
and $\{\h x^k\}$ converges to a limiting lifted stationary point of (\ref{PForm}).
\end{theorem}

\begin{proof} First, analogous to the proof to (\ref{psides1}),  one can show that for all $k \geq 0$
\begin{align*} & \ {\tilde\Psi}({\h x}^{k+1},{\h z}^k,{\h u}^k, \delta) +\tilde\theta_k\left[f(K{\h x}^k)- (\langle K{\h x}^{k+1},{\h y}^{k+1} \rangle -f^*({\h y}^{k+1}) )\right]\nn\\
 \le & \ {\tilde\Psi}({\h x}^{k},{\h z}^k,{\h u}^k, \delta) -\frac{\delta-L_{\nabla h}}{2}\|{\h x}^{k+1}-{\h x}^k\|^2. \vspace{-0.1cm}\end{align*}

Second, using that $g^*$ is  strongly convex with modulus $\frac{1}{\ell}$ and the optimality condition of ${\h z}^{k+1}$ in (\ref{Z-FSPSO}), it holds for all $k \geq 0$
\begin{align*} \langle A{\h x}^{k+1}, {\h z}^{k+1}\rangle-g^*({\h z}^{k+1}) \le \ & \langle A{\h x}^{k+1}, {\h z}^{k}\rangle-g^*({\h z}^{k})+\langle{\h z}^{k+1}-{\h z}^k, A{\h x}^{k+1}-A{\h x}^k\rangle \nn\\
& -\frac{1}{2\ell}\|{\h z}^k-{\h z}^{k+1}\|^2.\end{align*}
So,
%\begin{eqnarray*}
${\tilde\Psi}({\h x}^{k+1},{\h z}^{k+1},{\h u}^{k}, \delta)\le{\tilde\Psi}({\h x}^{k+1},{\h z}^{k},{\h u}^{k}, \delta)+\ell\|A\|^2\|{\h x}^{k+1}-{\h x}^k\|^2-\frac{1}{4\ell}\|{\h z}^k-{\h z}^{k+1}\|^2$.
%\end{eqnarray*}

Third,  analogous to the proof to (\ref{PriTheo:Phiu}), we have for all $k \geq 0$
\begin{eqnarray*}
{\tilde\Psi}({\h x}^{k+1},{\h z}^{k+1},{\h u}^{k+1}, \delta)\le{\tilde\Psi}({\h x}^{k+1},{\h z}^{k+1},{\h u}^{k}, \delta)-\frac{\delta(2-\beta)}{2\beta}\|{\h u}^k-{\h u}^{k+1}\|^2.
\end{eqnarray*}
Combining these three relations gives us that for all $k \geq 0$
\begin{align*}
 & \ {\tilde\Psi}({\h x}^{k+1},{\h z}^{k+1},{\h u}^{k+1}, \delta)-\tilde\theta_k f(K{\h x}^{k+1})\nn\\
\le &
-{\tilde c}_1\|{\h x}^k-{\h x}^{k+1}\|^2-{\tilde c}_2\|{\h u}^k-{\h u}^{k+1}\|^2
-{\tilde c}_3\|{\h z}^{k}-{\h z}^{k+1}\|^2,
 \end{align*}
 where ${\tilde c}_{1}:=\nu$, ${\tilde c}_{2}:=\frac{\delta(2-\beta)}{2\beta}$ and ${\tilde c}_3:=\frac{1}{4\ell}.$
  The remaining proofs are similar to Theorems \ref{subsequentialtoStat}, \ref{subsequential}, and  \ref{theo8}, thus omitted here.
  \end{proof}

Another interesting question is to see {\it what happens if we replace the updating step of ${\h z}^{k+1}$ in the conceptual algorithm FSPS
(\ref{Z-FSPSO})
 with the following step }
\begin{equation}\label{ProxZStep}
{\h z}^{k+1} = \arg\min_{\h z} \left\{ g^*(\h z) - \langle A{\h x}^{k+1}, {\h z} \rangle + \frac{\alpha_k}{2} \|\h z- {\h z}^k\|^2 \right\},
\end{equation}
where $\alpha_k>0$. We call this variant  as P-FSPS.
First, we see that if \( g \) is nonsmooth, then P-FSPS can also exhibit a cycling phenomenon, as illustrated by the following example.

      \begin{example}\label{example72}
Consider the problem (\ref{PForm}) for  ${\cal S}={[0,1]^2} \subseteq \R^2$, $A=\frac{1}{2}I,\; K=2 I$ where $I$ is the identity mapping, and $g,h,f:\mathbb{R}^2 \rightarrow \mathbb{R}$ given by
$g({\h x}) = \|{\h x}\|_{1}$, $h({\h x}) = \frac{1}{2}\|{\h x}\|^2+\frac{1}{2}{\h e}^\top {\h x}+\frac{3}{2}$ and  $f({\h x}) = \frac{1}{2}{\h e}^\top {\h x}+1$. Let $\beta=1$,   $\delta_k \equiv \delta = \frac{1}{2}$ for $k \geq 0$,  ${\tilde\theta_0}:=\frac{3}{2}$ and ${\h z}^0 =(1,1)^\top,\;{\h u}^0 = {\h x}^0:= (1,0)^\top$.
For any $\alpha_k>0$, P-FSPS generates a sequence $({\h x}^k, {\h y}^k, {\h z}^k, {\h u}^k)$ such that for all $k \geq 0$
$$ {\h u}^k = {\h x}^k=\left\{\begin{array}{rl} (0,1)^\top, & \mbox{if} \ k \ \mbox{is odd},\\
(1,0)^\top, & \mbox{if} \ k \ \mbox{is even}, \end{array}
\right. \quad {\h y}^k=\frac{1}{2}{\h e}, \; \; {\h z}^k ={\h e} \quad \mbox{and} \quad {\tilde\theta_k}=\frac{3}{2}.$$
Direct verification shows that neither $(1,0)^\top$ nor $(0,1)^\top$ is a limiting lifted stationary point of Example \ref{example72}.
\end{example}
We also note that, if $g$ is $\ell$-smooth, one can also develop an adapted version of P-FSPS such that it converges to an exact limiting lifted stationary point under the same assumptions in Theorem \ref{Excttheo8}. We omit the details for simplicity's sake.

\vspace{-0.1cm}
%\end{enumerate}

\section{Numerical results}\label{sec8}
We present numerical results to demonstrate the efficacy of the proposed algorithmic framework. All algorithms are implemented using MATLAB R2016a and executed on a desktop running Windows 10 equipped with an Intel Core i7-7600U CPU processor (2.80GHz) and 16GB of memory. Subsequently, all numerical results are conducted for the reformulated problems as indicated in Remark \ref{constr}.

\vspace{-0.1cm}

\subsection{Implementation details}\label{subsec71}

 In Algorithm \ref{subx:admm:Bes}, prior knowledge of $L_{\nabla h}$ is required, and the selection of $\delta_{k}$ might be conservative.
To enable larger step sizes $\delta_{k}$ and reduce the reliance on the unknown parameter $L_{\nabla h}$, we propose an adaptation of Algorithm \ref{subx:admm:Bes} by integrating a nonmonotone line search strategy \cite{WNF09}, denoted as Adaptive FSPS-{nls} (Algorithm \ref{A-FSPS-MLS}).
%In particular, the input parameters $L_{\nabla h}$ and $\nu$ have a minimal impact on the performance of Adaptive FSPS-nls due to the nonmonotone line search which aims to identify a suitable $\delta_k$ in each iteration.

{\small \begin{algo}[Adaptive FSPS algorithm  with {\it nonmonotone line search}]\label{A-FSPS-MLS} Let $0<\beta<2$, $\nu>0$, $0<q<1$, $\delta_0>0$, $\gamma_0=1$, and $\varepsilon>0$, $\eta > 1$, $0<\mu<1,\; c>0$, $T, \ell, t\in{\mathbb N}$. Let $({\h x}^0, {\h u}^0)$ be a given starting point.  We use {\text{\rm MaxIt}} to indicate the maximal number of iterations.
\vspace{-0.3cm}
\begin{align*}
        &\text{For}\  k=0 : {\rm MaxIt} \ \text{do}\\
		& \ \ \ \text{Set}\  \gamma_{k,0}:=\gamma_k.\\
		& \ \ \ \text{For} \ j = 0:\ell -1 \ \text{do}\\
        & \ \ \ \ \ \ \ \ \text{Set} \ \gamma_{k,j} :=\gamma_{k,0} q^j.\\
		& \ \ \ \ \ \ \ \ {Set} \ {\h{z}}^{k+1,j}:=\prox_{g^*/\gamma_{k,j}}\left(\frac{A{\h x}^{k}}{\gamma_{k,j}}\right).\\
		& \ \ \ \ \ \ \ \ \text{If}\  \theta_{k+1}:=\frac{{\rm\Psi}({\h x}^{k},{\h z}^{k+1,j},{\h u}^{k},\delta_k,\gamma_{k,j})}{f(K{\h x}^{k})}>0, \ \text{then}\\
		& \  \ \ \ \ \ \ \ \ \ \ \ \text{Update} \ \gamma_{k} :=\gamma_{k,j},\  {\h z}^{k+1}:={\h z}^{k+1,j}.\\
         & \ \ \ \ \ \ \ \ \ \ \ \ \text{Break}\\
        & \ \ \ \ \ \ \ \ \text{End If}\\
        &\ \ \ \text{End For} \\
		&\ \ \ \text{Set}\ \delta_{k,0} := 2\nu +L_{\nabla h}+ \frac{2\|A\|^2}{\gamma_{k}}.\\
		&\ \ \ \text{Choose}\ {\h y}^{k+1}\in \partial f(K{\h x}^k).\\
		&\ \ \ \text{Set} \ {\h d}^{k+1} := \theta_{k+1}K^* {\h y}^{k+1}-\nabla h({\h x}^k)-A^* {\h z}^{k+1}.\\
		& \ \ \ \text{For} \  s = 0 : t-1\\
		& \ \ \ \ \ \ \ \ \text{Set} \  \delta_{k,s} := \mu\eta^{s}\delta_{k,0}.\\
		&\ \ \ \ \ \ \  \ \text{Set} \ \tilde{\h x}^{k+1}:={\text{\rm Proj}}_{\cal S}\left({\h u}^k+\displaystyle{\frac{\h d^{k+1}}{\delta_{k,s}}}\right).\\
		&\ \ \ \ \ \ \ \  \text{If} \ F(\tilde{\h x}^{k+1})\leq \max\limits_{[k-T]_{+}\leq j\leq k}F(\h x^{j}) - \frac{c}{2}||\h x^{k}-\tilde{\h x}^{k+1}||^{2},  \ \text{then}\\
		&\ \ \ \ \ \ \ \ \ \ \  \ \text{Update} \  \h x^{k+1} := \tilde{\h x}^{k+1}.\\
        & \ \ \ \ \ \ \ \ \ \ \ \ \text{Break}\\
		& \ \ \ \ \ \ \ \ \ \text{End If}.\\
		& \ \ \ \text{End For}.\\
		&\ \ \ \text{Update}\ {\h u}^{k+1}:={\h u}^k-\beta({\h u}^k-{\h x}^{k+1}).\\	
        &\ \ \ \text{Update}\ \gamma_{k+1}:=\gamma_k,\ \delta_{k+1}:= \delta_k.\\
        &\ \ \ \ \text{If}\ \|{\h z}^{k+1}\|>\min\left(\frac{\varepsilon}{\gamma_{k}},\sqrt{\frac{2\varepsilon}{\gamma_{k}}}\right),\ \text{then}\\
		&\ \ \ \ \ \ \ \ \text{Update}\ \gamma_{k+1} :=\gamma_k q,\  \delta_{k+1} := 2\nu +L_{\nabla h}+ \frac{2\|A\|^2}{\gamma_{k+1}}.\\
		&\ \ \ \ \text{End If}\\
&\text{End For}
%	\end{algorithmic}
\end{align*}
\end{algo}}

\vspace{-0.1cm}
\subsection{Limited-angle CT reconstruction}\label{subsec72}

We solved the  limited-angle CT reconstruction problem  (\ref{CTmodel}) by comparing Adaptive FSPS-nls with the Extrapolated Proximal Subgradient algorithm (e-PSA) from \cite{BDL}, and the Proximity Gradient Subgradient algorithm with Backtracked Extrapolation (PGSA\_{BE}) from \cite{LSZ}. We set $\tau=0.1$ and $p=2$ in (\ref{CTmodel}) throughout the numerical tests. Each algorithm was initialized with the zero vector (with a safeguard mechanism of computing the denominator of  (\ref{PForm})
 via $\max(\|\nabla{\h x}\|_2,{\text{eps}})$) and used the same  stopping criterion defined by:
\begin{eqnarray}\label{stopC}
\vspace{-0.16cm}
\frac{\|{\h x}^{k+1}-{\h x}^k\|}{\max\{{\text{eps}},\|{\h x}^k\|\}}< 10^{-6} \;\;\mbox{or}\;\;k> {\text{MaxIt}}, \vspace{-0.2cm}\end{eqnarray}
where ${\text{eps}}$ represents the machine precision. We also adopted a {\it two-stage approach with a warm start strategy}, where the last iterate of the first stage served as the initial point for the second stage.
\footnote{The rationale behind employing a two-stage approach with different maximum iteration numbers was twofold: the initial stage, with a smaller maximum iteration limit, aims to explore a wider search space to identify a potential vicinity that contains a promising stationary point.  The subsequent stage, characterized by a larger maximum iteration number, focuses on a more refined local search, striving for increased precision in pinpointing the stationary point.} The warm start will be beneficial for solving the imaging processing problem, but it clearly also requires careful parameter tuning for two phases.

When implementing Adaptive FSPS-nls,
we set $f(\h {x}) := \| \h {x} \|_2$, $A = K = \nabla$, and followed Remark \ref{constr} by setting $g(\h {x}): = \tau \| \h {x} \|_1 + \frac{s}{2} \| \h {x} \|^2$ and $h(\h {x}) := \frac{1}{2} \| P \h {x} - \h {f} \|^2 - \frac{s}{2} \| A \h {x} \|^2$ where $s = 0.1$.
 The superscripts $^{(1)}$ and $^{(2)}$ represent the stage one and stage two, respectively. The parameter settings were
 $\beta^{(1)} = 1.1$, $\beta^{(2)} = 1.45$, $\nu^{(1)} = 2010$, $\nu^{(2)} = 0.1$, $\mu^{(1)} =\mu^{(2)}= 0.4$, $\eta^{(1)}=\eta^{(2)} = 1.5$, $q^{(1)}=q^{(2)} = 0.999$, $T^{(1)}=T^{(2)}  = 5$,
 $c^{(1)} =c^{(2)} = 1\text{e-4}$, $t^{(1)}=t^{(2)} = 250$, $\ell^{(1)}=\ell^{(2)}  = 1000$, $\text{MaxIt}^{(1)} = 50$,
       $\text{MaxIt}^{(2)} = 5000$, and $\varepsilon^{(1)}=\varepsilon^{(2)}=1e-6$.

When applying PGSA\_BE (Algorithm 1 in \cite{LSZ}), we set
$
f(\h{x}) := \tau \| \nabla \h{x} \|_1,\  h(\h{x}): = \frac{1}{2} \| P \h{x} - \h{f} \|^2, \ g(\h{x}) := \| \nabla \h{x} \|_2.
$
The inner loop amounts to solving in each iteration \vspace{-0.2cm}
\begin{eqnarray}\label{pgsasub}
\h{x}^{k+1} = \arg\min_{\h{x} \in \mathcal{B}} \left[ \tau \| \nabla \h{x} \|_1 + \frac{1}{2\alpha} \| \h{x} - \h{q}^k \|^2 \right],
\end{eqnarray}
with \( \h{q}^k = \h{u}^{k+1} - \alpha P^* (P \h{u}^{k+1} - \h{f}) + \alpha c_k \frac{\nabla^* (\nabla \h{x}^k)}{\| \nabla \h{x}^k \|_2} \),
${\h u}^{k+1}={\h x}^k+\beta_k({\h x}^k-{\h x}^{k-1})$ and \( c_k = \frac{f(\h{x}^k) + h(\h{x}^k)}{g(\h{x}^k)} \).
We applied ADMM to (\ref{pgsasub}) by introducing  \( \nabla \h{x} = \h{y} \) and \( \h{x} = \h{z} \), and with  \( \rho_1 \) and \( \rho_2 \) being the penalty parameters. For the outer loop parameters we set
$
\ell^{(1)} = \ell^{(2)} = 0, \beta_k^{(1)} = \beta_k^{(2)} \equiv 0.1,  \alpha^{(1)} = 0.0015, \alpha^{(2)} = 0.001,
$
$\varepsilon^{(1)} = \varepsilon^{(2)} = 1\text{e-3} \text{ (in the backtracking condition)}$,
$\text{MaxIt}^{(1)} = 50, \ \text{MaxIt}^{(2)} = 5000.$
For the inner loop parameters we set $\text{Inner\_tol}^{(1)} = \text{Inner\_tol}^{(2)} = 1\text{e-6},$
and $
\text{Inner\_MaxIt}^{(1)} = 1000, \;  \text{Inner\_MaxIt}^{(2)} = 200,  \rho_1^{(1)} = \rho_1^{(2)} = 1\text{e-4}, \ \rho_2^{(1)} = \rho_2^{(2)} = 1\text{e-2}$.

When applying e-PSA (Algorithm 4.1 in \cite{BDL}), we set $f^n(\h {x}) := \tau \| \nabla \h {x} \|_1$, $f^s(\h {x}): = \frac{1}{2} \| P \h {x} - \h {f} \|^2$, and $g(\h {x}) := \| \nabla \h {x} \|_2$. Due to the absence of boundedness condition (BC),  $\bar\mu=\bar\kappa=0$,
 and so, $\kappa_k=\mu_k=0$ and ${\h u}^k={\h v}^k={\h x}^k$ for all $k \geq 0$.
 The inner loop amounts to solving in each iteration
\begin{eqnarray}\label{epsasub}
\h{x}^{k+1} = \arg\min_{\h{x} \in \mathcal{B}} \left[ \tau \| \nabla \h{x} \|_1 + \frac{1}{2\tau_k} \| \h{x} - \h{p}^k \|^2 +\frac{\ell}{2}\|{\h x}-{\h x}^k\|^2\right],
\end{eqnarray}
with $\h p^k={\h x}^k+\tau_k\theta_k \frac{\nabla^* \nabla {\h x}^k}{\|\nabla \h x^k\|}-\tau_k P^*(P{\h x}^k-{\h f})$
and \( \theta_k = \frac{f^n(\h{x}^k) + f^s(\h{x}^k)}{g(\h{x}^k)} \).
We solved (\ref{epsasub}) also via ADMM, with  \( \rho_1 \) and \( \rho_2 \) being the penalty parameters. For the outer loop parameters we set $\beta^{(1)}=\beta^{(2)}=0$, $\ell^{(1)}=\ell^{(2)}=\|P{\h x}^{\text{true}}\|/\|{\h x}^{\text{true}} \|$,
$\tau_k^{(1)}=\tau_k^{(2)}\equiv760$, and $\text{MaxIt}^{(1)} = 50, \ \text{MaxIt}^{(2)} = 5000.$
For the inner loop parameters we set $\text{Inner\_tol}^{(1)} = \text{Inner\_tol}^{(2)} = 1\text{e-6},$
and $\text{Inner\_MaxIt}^{(1)} = 1000,\ \text{Inner\_MaxIt}^{(2)} = 200,  \ \rho_1^{(1)} = \rho_1^{(2)} =\rho_2^{(1)} = \rho_2^{(2)} = 1\text{e-2}.$

We assessed performance based on two metrics:
the root mean squared error (RMSE) \cite{WTNL} and
the overall structural similarity index (SSIM) \cite{WBSS}.
We conducted tests on parallel beam CT reconstruction of the Shepp-Logan phantom using projection ranges of \(90^\circ\), \(120^\circ\), and \(150^\circ\). We evaluated both noiseless and noisy scenarios where the Gaussian noise with zero mean and standard deviation (SD) \(\sigma=0.001, \ 0.005\).  The performance of the three algorithms is summarized in Table \ref{Tab2}. The results of the computation indicate that the adaptive FSPS-nls has better performance in terms of SSIM, RMSE and CPU time (in seconds) compared to the recently introduced double-loop algorithms PGSA\_BE and e-PSA.

\begin{sidewaystable}
\caption{Parallel beam CT reconstruction of the Shepp-Logan phantom for different projection ranges.}
\label{Tab2}
\centering
\begin{tabular}{c|c|ccc|ccc|ccc|cc}
\hline
\multirow{2}{*}{SD of}& \multirow{2}{*}{Range} & \multicolumn{3}{c|}{e-PSA \cite{BDL}}  & \multicolumn{3}{c|}{PGSA\_BE \cite{LSZ}}               & \multicolumn{3}{c|}{Adaptive FSPS-nls}        & \multicolumn{2}{c}{Time Ratios}       \\
\multirow{2}{*}{ \small noise}& \multirow{3}{*}{}           & \multicolumn{3}{c|}{}                  &\multicolumn{3}{c|}{}   &\multicolumn{3}{c|}{} & \multicolumn{2}{c}{}\\
\cline{3-13}
 &               & SSIM   & RMSE     & T$_1$     & SSIM & RMSE      & T$_2$    & SSIM   & RMSE     & T$_3$   &T$_3$/T$_1$ & T$_3$ / T$_2$ \\
\hline
\multirow{3}{*}{0}
 & 90$^{\circ}$  & 0.9946 & 3.38e-04 & 2.95e+03 & 0.9958 & 3.65e-04 & 3.09e+03 & 1.0000 & 2.39e-05 & 3.67e+01 & 1.25\% & 1.19\% \\
 & 120$^{\circ}$ & 0.9970 & 1.98e-04 & 2.85e+03 & 0.9999 & 3.37e-05 & 2.51e+03 & 1.0000 & 1.21e-05 & 3.28e+01 & 1.15\% & 1.31\% \\
 & 150$^{\circ}$ & 0.9991 & 1.06e-04 & 2.71e+03 & 1.0000 & 1.95e-05 & 1.39e+03 & 1.0000 & 8.47e-06 & 3.17e+01 & 1.17\% & 2.28\% \\
\hline
\multirow{3}{*}{0.001}
 & 90$^{\circ}$  & 0.9946 & 3.38e-04 & 2.95e+03 & 0.9955 & 3.76e-04 & 1.64e+03 & 0.9999 & 2.94e-05 & 1.32e+01 & 0.45\% & 0.8\% \\
 & 120$^{\circ}$ & 0.9970 & 1.97e-04 & 1.80e+03 & 0.9999 & 3.55e-05 & 1.47e+03 & 1.0000 & 1.49e-05 & 1.41e+01 & 0.78\% & 0.96\% \\
 & 150$^{\circ}$ & 0.9991 & 1.08e-04 & 2.99e+03 & 1.0000 & 2.27e-05 & 1.40e+03 & 1.0000 & 1.10e-05 & 1.53e+01 & 0.51\% & 1.10\% \\
\hline
\multirow{3}{*}{0.005}
 & 90$^{\circ}$ & 0.9944  & 3.21e-04 & 2.07e+03 & 0.9865 & 6.22e-04 & 1.62e+03 & 0.9992 & 1.12e-04 & 4.08e+01 & 1.97\% & 2.52\% \\
 & 120$^{\circ}$ & 0.9967 & 2.15e-04 & 3.02e+03 & 0.9994 & 9.83e-05 & 2.18e+03 & 0.9996 & 8.65e-05 & 1.02e+02 & 3.39\% & 4.71\% \\
 & 150$^{\circ}$ & 0.9984 & 1.54e-04 & 1.14e+03 & 0.9996 & 8.35e-05 & 1.42e+03 & 0.9996 & 8.13e-05 & 1.30e+01 & 0.92\% & 1.15\% \\
\hline
\end{tabular}
\end{sidewaystable}

\subsection{Robust Sharpe ratio type minimization problem}\label{subsec73}

 We tested Adaptive FSPS-nls also on the robust sharp-ratio minimization problem (\ref{muleq}), and compared it with
PGSA\_BE,  e-PSA  and the
Dinkelbach's  method with Surrogation (DLS) \cite[Algorithm 7.2.7]{CuiPang}.
 The data $((r_i)_{i=1}^{m_1},({\h a}_i)_{i=1}^{m_1},{(C_i)}_{i=1}^{m_2})$ were generated as follows:
(1) each vector ${\h a}_i$ was generated such that each entry is drawn from a uniform distribution over the interval $[0,1]$;
(2)  ${r_i}$ was set to be greater than $\|{\h a}_i\|_{\infty}$;
(3) each matrix $C_i$ was generated such that each eigenvalue conforms to a uniform distribution over the interval $[10^{-3},1+10^{-3}]$.

We measured these algorithms's performance based on several metrics: the objective function value ${\tt obj}$, the infeasibility ${\tt infeas}:=\|\max(-{\h x},0)\|_1+|\|\h x\|_1-1|$, and the lifted limiting stationarity residual
$${\tt stat}:={\text{dist}}\big( {\bf 0}, \ \left(A^* \partial g(A{{\h x}})+\nabla h({\h x})+\partial \iota_{\cal S}({{\h x}})\right) f(K{\h x})-(g(A{\h x})+h({\h x}))K^* \partial f(K{{\h x}})\big),$$
 evaluated at the last iterate.
 % For the computation of the {\rm pKKT}, we solve the minimization problem $\min_{\xi\in N_{\cal S}(\h x)} \|\nabla F_{i,j}(\h x)+\xi\|$ to obtain the optimal value.

We also used (\ref{stopC}) as a stopping criterion.
 When implementing  Adaptive FSPS-nls, we set $f$, $\h r$, $A$, and $K$ as in Section \ref{sec1}, and set $g(\h x):=\|{\h r}-{\h x}\|_{\infty}+\frac{s}{2}\|{\h x}\|^2$ and $h(\h x):=-\frac{s}{2}\|A{\h x}\|^2$ with $s=0.01$, by following Remark \ref{constr}. We set the algorithm parameters as $\ell:=100$, $L_{\nabla h}:=s \|A^*A\|$, $\nu:=1$, $\eta:=1.15$, $q:=0.999$, $\mu:=0.005$, $c:=10^{-4}$, $T:=5$, $\delta_0:=2\nu + L_{\nabla h}+2\|A\|^2$, $t:=250$, ${\text{\rm MaxIt}} := 500$, $\varepsilon:=1e-6$, and  $\beta:=1.6$.
When implementing PGSA\_BE, we defined
$
f(\h{x}) := \max_{1 \le i \le m_1}\{r_i - \h{a}_i^\top \h{x}\}, h({\h x}):=0$, and $g(\h{x}) :=\max_{1 \le i \le m_2} \h{x}^\top C_i \h{x}.
$
The inner loop amounts to solving in each iteration
\begin{eqnarray}\label{SRsub}
\vspace{-0.13cm}
\h{x}^{k+1} = \arg\min_{\h{x} \in  \Delta}\left[ \max_{1\le i\le m_1} \{r_i-{\h a}_i^\top {\h x}\}+\frac{1}{2\alpha}\|{\h x}-{\h u}^{k+1}-\alpha c_k {\h y}^k\|^2\right],\vspace{-0.16cm}
\end{eqnarray}
with ${\h y}^k\in \partial g(\cdot)({\h x}^k)$ and ${\h u}^{k+1}={\h x}^k+\beta_k({\h x}^k-{\h x}^{k-1})$. The inner loops of both  e-PSA and DLS amounts to solving in each iteration a similar problem as (\ref{SRsub}).

For fair comparisons, we solved the inner loop subproblems for all these double-loop algorithms via ADMM. We used
\(\text{MaxIt} = 500\) for both inner and outer loops. In addition, we used for
e-PSA as outer loop parameters
$\beta = 0, \ \tau_k \equiv 0.5$ and as
inner loop parameters
$\rho_1 = 0.01, \ \rho_2 = 0.5$; we used for
PGSA\_BE as outer loop parameters
$\beta_k \equiv 0.01, \ \alpha = 0.1, \ \varepsilon = 1 \times 10^{-3} \text{ (in the backtracking)}$ and as inner loop parameters
$\rho_1 = 0.01, \ \rho_2 = 0.5$, and we used for DLS as inner loop parameters
$\rho_1 = \rho_2 = 0.01.$
We conducted numerical tests by setting $(n, m_1, m_2)$ to $(100, 5, 20)$, $(100, 20, 5)$, $(100, 20, 20)$, $(400, 20, 10)$, $(400, 10, 20)$, and $(400, 20, 20)$. We performed 50 trials for each configuration. The average values of the considered performance metrics, along with the CPU time (in seconds), are reported in Table \ref{Spop}.

As  observed, FSPS-{nls} demonstrates superior performance compared to  e-PSA, PGSA\_{BE} and DLS. Specifically, it achieves a smaller ${\tt stat}$ than PGSA\_{BE} and e-PSA, while attaining a comparable ${\tt stat}$ to DLS. Moreover, FSPS-{nls} yields a smaller ${\tt infeas}$ and requires less computation time, while achieving  comparable ${\tt obj}$ values.

\begin{table}[t]
	\centering
	{ \vskip -2mm\caption{Adaptive FSPS-nls versus double-loop algorithms for robust sharp-ratio type optimization  Problem}\label{Spop}
	\vskip -2mm\begin{tabular}{cccccc}
		\hline
		\multirow{2}{*}{$(n, m_1,m_2)$ } &     & \multirow{2}{*}{FSPS-nls}   & \multirow{2}{*}{e-PSG}     & \multirow{2}{*}{PGSA\_BE}& \multirow{2}{*}{DLS}\\
		                                 &                                &                           &                  &   &\\
		 \hline
		\multirow{4}{*}{\begin{tabular}[c]{@{}c@{}}\\ (100, 5, 20)\end{tabular}}
		 &{\tt obj}    & 1.52e+00 	 	 & 1.36e+00     & 1.24e+00  & 1.50e+00\\
         &{\tt infeas} & 4.87e-09 	  	 & 3.07e-07     & 5.37e-07 	& 4.28e-07\\
         &{\tt stat}   & 2.53e-07        & 1.64e-02     & 1.18e-02  & 1.66e-07\\
         &{ CPU}    & 2.82e-02 	     & 1.06e+00 	& 4.74e+00  & 3.33e-02\\

\hline
		\multirow{4}{*}{\begin{tabular}[c]{@{}c@{}} \\ (100, 20, 5)\end{tabular}}
   &{\tt obj}         &1.76e+00           &1.79e+00 	&1.78e+00 	&1.74e+00 \\
   &{\tt infeas}      &5.90e-09 		  &5.23e-07     &5.25e-07   &4.01e-07 \\
   &{\tt stat}        &3.20e-07  	 	  &6.34e-03 	&2.35e-07 	&3.06e-07\\
   &{ CPU}         &1.11e-02 		  &1.03e+00 	&3.83e+00   &3.31e-02\\

\hline		
		\multirow{4}{*}{\begin{tabular}[c]{@{}c@{}}\\ (100, 20, 20)\end{tabular}}
 &{\tt obj}         &1.68e+00 	  	    &1.62e+00 	 &1.59e+00 	&1.62e+00\\
 &{\tt infeas}      &4.23e-09 	  	    &4.20e-07 	 &5.25e-07 	&3.51e-07\\
 &{\tt stat}        &4.07e-07  	        &6.27e-03 	 &1.90e-07 	&3.13e-07\\
 &{ CPU}         &1.36e-02 	 	    &1.04e+00 	 &4.00e+00 	&5.34e-02\\

\hline
	\multirow{4}{*}{\begin{tabular}[c]{@{}c@{}}\\ (400, 20, 10)\end{tabular}}
 &{\tt obj}          &1.88e+00 	    &1.84e+00 	&1.85e+00 	&1.82e+00\\
 &{\tt infeas}       &3.24e-09 	 	&4.40e-07 	&3.43e-07 	&4.52e-07\\
 &{\tt stat}         &6.30e-05    	&6.33e-05 	&6.57e-05 	&5.60e-05\\
 &{ CPU}          &2.68e-01 	  	&4.45e+00 	&1.64e+01 	&7.63e-01\\

 \hline
	\multirow{4}{*}{\begin{tabular}[c]{@{}c@{}}\\ (400, 10, 20)\end{tabular}}
 &{\tt obj}           &1.70e+00 	 	&1.71e+00   &1.66e+00 	&1.66e+00\\
 &{\tt infeas}        &5.40e-09 	 	&5.30e-07 	&3.89e-07 	&4.71e-07\\
 &{\tt stat}          &3.02e-05  	 	&5.41e-03 	&2.39e-05 	&2.61e-05\\
 &{ CPU}           &4.60e-01 		&6.99e+00 	&2.36e+01 	&8.35e-01 \\
 \hline
	\multirow{4}{*}{\begin{tabular}[c]{@{}c@{}}\\ (400, 20, 20)\end{tabular}}
	 &{\tt obj}           &1.84e+00 	  	&1.78e+00 	 &1.78e+00 	 &1.83e+00\\
     &{\tt infeas}        &4.63e-09 	 	&3.15e-07 	 &3.67e-07 	 &4.34e-07\\
     &{\tt stat}          &4.32e-05  	 	&3.39e-05 	 &3.65e-05 	 &3.62e-05\\
     &{ CPU}           &4.68e-01 	 	&7.18e+00 	 &2.13e+01 	 &1.48e+00\\

\hline
	\end{tabular}}
\end{table}

\section{Conclusions}\label{sec9}

 The paper focuses on a class of fractional programs characterized by linear compositions with nonsmooth function in both the numerator and denominator. We develop a single-loop full-splitting algorithm, Adaptive FSPS, to overcome the challenges in computing the proximal point of the linear composition with the nonsmooth component in the numerator. Our contributions include establishing the global convergence of the proposed algorithm toward an exact/approximate lifted stationary point under the Kurdyka-\L ojasiewicz (KL) property when \( g \) is smooth/nonsmooth, without imposing full-row rank assumptions on the linear operators. We explain the rationale behind the convergence to an approximate lifted stationary point and construct counterexamples to show that pursuing an exact solution can lead to divergence when \( g \) is nonsmooth. We propose a practical version of Adaptive FSPS-nls that incorporates a nonmonotone line search to improve the performance. Finally, we demonstrate the superiority of Adaptive FSPS-nls over some existing state-of-the-art methods for solving two concrete applications.

\appendix
\section{Any  accumulation point of the sequence generated by the FSPS may fail to be a lifted stationary point even when  $\gamma_k \downarrow 0$}\label{counter1}
%{\underline{\it Counter-example 1}}

 Consider the counter-example of Example \ref{ex6dot10}.
 For \(\beta = 1\), \(\gamma_k := \frac{1}{k+1}\), \(\delta_k \equiv 1\), \(\theta_0 := 2\), and \({\h z}^0 = {\h u}^0 = {\h x}^0 := (1, 0)^\top\), FSPS generates a sequence \(\{({\h x}^k, {\h y}^k, {\h z}^k, {\h u}^k)\}\).
The sequence $\{\h x^k\}$ has two accumulation points: \((1, 0)^\top\) and \((0, 1)^\top\).
Indeed, neither \((1, 0)^\top\) nor \((0, 1)^\top\) is a limiting lifted stationary point of Example \ref{ex6dot10}. We provide the details in the following lemma.

\begin{lemma}\label{lemma1}
Let the sequence $\{({\h x}^k, {\h y}^k, {\h z}^k, {\h u}^k)\}$ be generated by  FSPS (\ref{FSPSO}) for solving Example \ref{ex6dot10}
with \(\beta = 1\), \(\gamma_k := \frac{1}{k+1}\), \(\delta_k \equiv 1\), \(\theta_0 := 2\), and \({\h z}^0 = {\h u}^0 = {\h x}^0 := (1, 0)^\top\).
Then, we have ${\h y}^k={\h e}$ for all $k \geq 1$ and the following statements hold:
\begin{itemize}
\item[(i)]
$$ \large{\textcircled{\small{1}}_k:} \; \; {\h z}^k =\left\{\begin{array}{rl} (0,1)^\top, & \mbox{if} \ k \ \mbox{is odd},\\
(1,0)^\top, & \mbox{if} \ k \ \mbox{is even}, \end{array}
\right. $$
\item[(ii)]
$$ \large{\textcircled{\small{2}}_k:} \; \;{\h x}^k =\left\{\begin{array}{rl} (0,\ \theta_{k-1}-1)^\top, & \mbox{if} \ k \ \mbox{is odd},\\
(\theta_{k-1}-1,\ 0)^\top, & \mbox{if} \ k \ \mbox{is even}. \end{array}
\right. $$
where $\theta_k$ is given by \begin{eqnarray}\label{kk}\;\;\;\; \large{\textcircled{\small{3}}_k:}\; \;\theta_{k}=\frac{\theta_{k-1}-1+(0.5*(\theta_{k-1}-1)^2+(\theta_{k-1}-1)+0.5)-\frac{1}{2k}}{\theta_{k-1}-1/2}.\end{eqnarray}
\item[(iii)] $\lim_{k \rightarrow \infty}\theta_k=2$, and hence, the sequence $\{\h x^k\}$ has two accumulation points: \((1, 0)^\top\) and \((0, 1)^\top\).
\end{itemize}
\end{lemma}
\begin{proof}  First, we define a sequence $\{b_k\}$ via the following recurrence formula: $b_0=1$ and for all $k \ge 0$ $b_{k+1}=\frac{\frac{1}{2}b_k^2+b_k-\frac{1}{2(k+1)}}{b_k+1/2}.$
For this sequence, we  first use {\it mathematical induction} to see that
\begin{eqnarray}\label{rangeb} 1/2<b_k <1, \;\forall k\ge 1.\end{eqnarray}
By direct calculations, we see that %$b_0=1$,
$b_1=2/3$, $b_2=23/42$, $b_3=936.5/1848$ and $b_4=\frac{936.5^2/(1848*2)+705.5}{1860.5}$. Thus, (\ref{rangeb}) holds with $k=1,2,3,4$. Suppose that (\ref{rangeb}) holds with $k=k_0$ for some $k_0 \ge 4$, that is, $1/2<b_{k_0}<1$. We now show that (\ref{rangeb}) holds with $k=k_0+1$. To see this,   we first note that  $b_{k_0+1}=\frac{1}{2}b_{k_0}+\frac{3}{4}-\frac{\frac{3}{8}+\frac{1}{2(k_0+1)}}{b_{k_0}+1/2}.
$
Define a one-variable function $f(x):=\frac{1}{2}x+\frac{3}{4}-\frac{\frac{3}{8}+\frac{1}{2(k_0+1)}}{x+1/2}$. Direct verification shows that
$f$
is an increasing function. So,
$$b_{k_0+1}=f(b_{k_0})\ge f(1/2)=5/8-\frac{1}{2(k_0+1)}>1/2,$$
where the last strict inequality holds as $k_0 \ge 4$.
Moreover, as $b_{k_0} < 1$,
$
b_{k_0+1}=f(b_{k_0})\le f(1) <1.
$
Thus, (\ref{rangeb})  holds.

Next, we show the main results of this lemma.
Clearly, from the definition of $f$ and the construction, ${\h y}^k={\h e}$ for all $k \geq 1$.

[Proof of {\rm (i)} \& {\rm (ii)}] We use  mathematical induction to verify $\large{\textcircled{\small{1}}_k}$, $\large{\textcircled{\small{2}}_k}$ and
$\large{\textcircled{\small{3}}_k}$ hold for all $k \ge 1$. A direct verification shows that the statements of $\large{\textcircled{\small{1}}_k}$, $\large{\textcircled{\small{2}}_k}$ and $\large{\textcircled{\small{3}}_k}$ hold for $k=1,2$; Suppose that $\large{\textcircled{\small{1}}_k}$, $\large{\textcircled{\small{2}}_k}$ and $\large{\textcircled{\small{3}}_k}$ hold
for $k \le k_0$ with $k_0 \ge 2$.
Using (\ref{rangeb}) with $b_{k_0}=\theta_{k_0}-1$, we see that
$3/2<\theta_{k_0}<2$. Using the update formula of $\h x^{k+1}$ in (\ref{FSPSO}), a direct verification shows that $\large{\textcircled{\small{2}}_k}$ holds with $k=k_0+1$. Note from the update formula of $\h z^{k+1}$ in (\ref{FSPSO}) that ${\h z}^{k_0+1}:={\text{\rm Proj}}_{{\cal B}^\infty_1}(\frac{(\theta_{k_0}-1,\ 0)^\top}{1/(k_0+1)})$ or
${\h z}^{k_0+1}:={\text {\rm  Proj}}_{{\cal B}^\infty_1}(\frac{(0,\ \theta_{k_0}-1)^\top}{1/(k_0+1)})$ where ${\cal B}^\infty_1$ is the unit ball defined by the $\ell_\infty$-norm.
Since $3/2<\theta_{k_0}<2$, we have
$\large{\textcircled{\small{1}}}_{k}$ holds with $k=k_0+1$. Finally, using the  update formula of $\theta_{k+1}$ in (\ref{FSPSO}),  $\large{\textcircled{\small{3}}}_{k}$ with $k=k_0+1$ also follows.

[Proof of {\rm (iii)}] To see {\rm (iii)}, we first establish that
$b_{k+1}\ge b_k \;\mbox{when}\;k\ge 4.$
From the definition of the sequence $\{b_k\}$, this is equivalent to
\begin{eqnarray}\label{eq:use19} \frac{1-\sqrt{1-\frac{4}{k+1}}}{2}\le b_k\le \frac{1+\sqrt{1-\frac{4}{k+1}}}{2}.\end{eqnarray}
Clearly, \eqref{eq:use19} is true with $k=4$ by direct computation. Suppose now \eqref{eq:use19} holds with $k=k_0$ with $k_0 \ge 4$. We now show that \eqref{eq:use19} holds with $k=k_0+1$, that is,
 \begin{eqnarray*}
 %\frac{1-\sqrt{1-\frac{4}{k+1}}}{2}\le b_k{\le } \frac{1+\sqrt{1-\frac{4}{k+1}}}{2}\Rightarrow
 \frac{1-\sqrt{1-\frac{4}{k_0+2}}}{2}\overset{(\clubsuit)}{\le } b_{k_0+1}\overset{(\spadesuit)}{\le} \frac{1+\sqrt{1-\frac{4}{k_0+2}}}{2}.\end{eqnarray*}
For $(\clubsuit)$, it holds obviously due to $b_{k}>1/2$ for all $k \ge 1$.  To prove $(\spadesuit)$, recall the one-variable $f$ defined as above. We have $b_{k_0+1}=f(b_{k_0}) \le f(\frac{1+\sqrt{1-\frac{4}{k_0+1}}}{2}),$ where the last inequality follows by the induction hypothesis and the fact that $f$ is increasing. Thus, it remains to show that $f(\frac{1+\sqrt{1-\frac{4}{k_0+1}}}{2})\le\frac{1+\sqrt{1-\frac{4}{k_0+2}}}{2}$.
 By letting $\delta:=\frac{1}{k_0+1}$, $\kappa:=\frac{1}{k_0+2}$ and noting that $\kappa\le\delta$.
 Let $c:=\sqrt{1-4\delta}$, $d:=\sqrt{1-4\kappa}$.
 thus we have
 $\frac{c^2}{4}-1\le \frac{1}{4}cd+\frac{d}{2}-\frac{c}{2}-1.$
 Consequently,
 $1+\frac{\sqrt{1-4\delta}}{4}-\frac{\frac{3}{8}+\frac{\delta}{2}}{1+\frac{1}{2}\sqrt{1-4\delta}}\le \frac{1}{2}+\frac{1}{2}\sqrt{1-4\kappa}.$
 With some elementary calculations, it leads to $f(\frac{1+\sqrt{1-\frac{4}{k_0+1}}}{2})\le\frac{1+\sqrt{1-\frac{4}{k_0+2}}}{2}$.
Therefore, the sequence of $\{b_k\}$ is monotone and bounded, thus $\lim\limits_{k\to+\infty}b_k$ exists.
Consequently,
 $\lim\limits_{k\to+\infty}\theta_k$ exists, and$\lim\limits_{k\to+\infty}\theta_k = 2$.
\end{proof}

\end{document}